\documentclass[reqno,12pt,letterpaper]{amsart}

\usepackage{style}

\begin{document}

\title[Existence of quasinormal modes for Kerr--AdS Black Holes]
{Existence of quasinormal modes for Kerr--AdS Black Holes}
\author{Oran Gannot}
\email{gannot@northwestern.edu}
\address{Department of Mathematics, Lunt Hall, Northwestern University,
Evanston, IL 60208, USA}

\begin{abstract}
This paper establishes the existence of quasinormal frequencies converging exponentially to the real axis for the Klein--Gordon equation on a Kerr--AdS spacetime when Dirichlet boundary conditions are imposed at the conformal boundary. The proof is adapted from results in Euclidean scattering about the existence of scattering poles generated by time-periodic approximate solutions to the wave equation.
\end{abstract}

\maketitle

\section{Introduction}

Recent years have seen substantial progress in the analysis of linear fields on asymptotically anti-de Sitter (aAdS) backgrounds. Understanding the boundedness and decay of solutions to the linear wave equation is a prerequisite for studying nonlinear (in)stability of aAdS spacetimes. Furthermore, linear fields play a distinguished role in the AdS/CFT correspondence. 

One common way of describing linear perturbations of black holes is through the quasinormal frequency (QNF) spectrum, consisting of complex frequencies $\lambda$ associated with oscillating and decaying quasinormal mode (QNM) solutions to the wave equation. Despite a sizable literature concerning the QNFs of aAdS black holes, there are still many open questions in the mathematical study of these objects.

One notable conjecture is that aAdS black holes should display QNFs rapidly converging to the real axis. This phenomenon was first observed for the Schwarzschild--AdS solution through numerical and formal WKB analysis \cite{dias:2012:cqg:b,festuccia2009bohr}. The existence of such weakly damped modes is consistent with at most logarithmic local energy decay in time for solutions of the wave equation \cite{gannot:2014,holzegel2013decay,holzegel2014quasimodes}. In addition, the asymptotic relationship between these QNFs and the spectrum of global AdS at high energies provides a link between the conjectured instabilities of global AdS and Kerr--AdS
\cite{balasubramanian2014holographic,bizon:2014:grg,bizon:2011:prl,bizon2015resonant,buchel2015conserved,craps:2014vaa,craps:2014jwa,dias:2012:cqg:b,dias:2012:cqg,holzegel2013decay}.

In \cite{gannot:2014}, the existence of QNFs converging exponentially to the real axis was rigorously established for Schwarzschild--AdS black holes. This was based on a construction of quasimodes (not to be confused with QNMs), which are time-periodic approximate solutions to the Klein--Gordon equation. Due to the spherical symmetry of the Schwarzschild--AdS metric, the wave equation separates into a family of one dimensional equations indexed by angular momenta $l$, each of which fits into the framework of classical scattering theory on the half-line. The quasimodes constructed in \cite{gannot:2014} reflect the existence of null-geodesics which are trapped between the conformal boundary and an effective potential barrier. By applying general results of Tang--Zworski \cite{tang1998quasimodes} from Euclidean scattering, it was possible to conclude the existence of QNFs converging exponentially to the real axis as $l \rightarrow \infty$.

For the rotating Kerr--AdS solution, it is more difficult to demonstrate the existence of long-lived QNMs. Because of the more complicated structure of the separated equations, a WKB analysis is harder to perform. Futhermore, the author is not aware of any numerical studies of QNFs for Kerr--AdS in the high frequency limit. 

Nevertheless exponentially accurate quasimodes have been constructed for Kerr--AdS metrics by Holzegel--Smulevici \cite{holzegel2014quasimodes} (at a linearized level, the construction is the same as in \cite{gannot:2014}); their motivation was to establish a logarithmic lower bound for energy decay. A natural question is whether the methods of Tang--Zworski can be adapted to deduce the existence of QNFs converging to the real axis from these quasimodes. This is accomplished here by establishing the following theorem. 

\begin{theo} [Main Theorem]\label{theo:maintheorem}
	Fix a cosmological constant $\Lambda < 0$, black hole mass $M> 0$, rotation speed $a \in \mathbb{R}$ satisfying $|a|^2 < 3/|\Lambda|$, and Klein--Gordon mass $\nu > 0$; the location of the horizon is at $r=r_+$. Let $X_\delta = (r_+ - \delta,\infty)\times \mathbb{S}^2$ for $\delta > 0$ sufficiently small, and $t^\star$ be the Kerr-star time coordinate. Then there exists a sequence of complex numbers and smooth functions 
	\[
	\lambda_\ell \in \mathbb{C}, \quad u_\ell \in C^\infty(X_\delta), \quad \ell \geq L
	\]
	for some $L \geq 0$ with the following properties.
	
	\begin{enumerate} \itemsep6pt
		\item \label{it:kg} The functions  $v_\ell = e^{-i \lambda_\ell t^\star}u_\ell$ solve the Klein--Gordon equation
		\[
		\Box_g v_\ell +\frac{|\Lambda|}{3}(\nu^2-9/4) v_\ell = 0.
		\]
		
		\item \label{it:width} The complex frequencies $\lambda_\ell$ satisfy
			\[
			\ell/C < \Re \lambda_\ell < C\ell, \quad 0 < - \Im \lambda_\ell < e^{-\ell/D}
			\]
			for some $C, D > 0$.
			
		\item \label{it:outgoing} Each $u_\ell$ is smooth up to $\{ r = r_+ - \delta\}$ and has a nonzero restriction to $\{ r > r_+\}$.
		
		\item \label{it:dirichlet}  Each $u_\ell$ satisfies
		\[
		\int_{X_\delta} |u_\ell|^2 \, r^{-1} \,dS_t < \infty,
		\] 
		where $dS_t$ is the surface measure induced on $X_\delta$ by $g$, and moreover 
		\[
		\lim_{r\rightarrow \infty} r^{3/2-\nu}u_\ell = 0.
		\]
		
		\item \label{it:axi} Each $u_\ell$ is axisymmetric in the sense that $D_{\phi}u =0$, where $\phi$ is the azimuthal angle on $\mathbb{S}^2$.
\end{enumerate}
\end{theo}

The frequencies $\lambda_\ell$ in Theorem \ref{theo:maintheorem} are QNFs, and $u_\ell$ are associated QNMs. This theorem is deduced from the existence of \emph{real} frequencies $\lambda_\ell^\sharp \in \mathbb{R}$ and functions $u^\sharp_\ell$ supported in $\{ r > r_+\}$ for which \eqref{it:width}, \eqref{it:dirichlet}, \eqref{it:axi} hold, and for which \eqref{it:kg} is approximately satisfied (see Theorem \ref{theo:quasimodes} below for a more precise statement regarding these quasimodes).

\begin{rem} \begin{inparaenum} \item The functions $v_\ell$ are smooth solutions to the Klein--Gordon equation in a region extending past the event horizon. This reflects the outgoing nature of QNMs. One also obtains a nonzero solution to the Klein--Gordon equation in the black hole exterior by restriction.
		
		\item The square integrability condition \eqref{it:dirichlet} constrains the growth of $u_\ell$ as $r\rightarrow \infty$. In fact, $u_\ell$ has an asymptotic expansion near the conformal boundary determined by the indicial roots of the Klein--Gordon operator \cite[Proposition 4.17]{gannot:bessel}. The coefficient of $r^{\nu - 3/2}$ vanishes, which is a type of Dirichlet boundary condition.

\end{inparaenum} \end{rem}

As will be clear from the proof, Theorem \ref{theo:maintheorem} is a \emph{black box} in the sense that any sequence of quasimodes satisfying the conditions of Theorem \ref{theo:quasimodes} can be plugged into the machinery to obtain a corresponding sequence of QNFs. Furthermore, there is a relationship
\begin{equation} \label{eq:close}
| \lambda_\ell - \lambda_\ell^\sharp | \leq e^{-\ell/C}
\end{equation}
for some $C > 0$, so any description of $\lambda_\ell^\sharp$ modulo $\mathcal{O}(\ell^{-\infty})$ gives a corresponding description for $\Re \lambda_\ell$ as $\ell \rightarrow \infty$. The imprecise localization of $\Re \lambda_\ell$ in Theorem \ref{theo:maintheorem} is therefore only due to the inexact nature of the quasimodes constructed in \cite{holzegel2014quasimodes}; this should be compared to the main theorem of \cite{gannot:2014} in the simpler Schwarzschild--AdS setting, where $\Re \lambda_\ell$ admits an asymptotic expansion in powers of $\ell^{-1/2}$.

\begin{rem} \begin{inparaenum} \item Exponential accuracy of the quasimodes is not necessary to deduce the existence of QNFs --- see the proof of Theorem \ref{theo:maintheorem}, as well as \cite{stefanov1999quasimodes,stefanov2005approximating,tang1998quasimodes} for more general results in the Euclidean setting. Less accurate quasimodes could potentially result in slower convergence to the real axis, as well as a weaker version of \eqref{eq:close}.
		
\item Theorem \ref{theo:maintheorem} still applies if the quasimodes are supported on finitely many eigenspaces of $D_\phi$ (uniformly in $\ell$).

\item Quasimodes satisfying more general self-adjoint boundary conditions also yield a version of Theorem \ref{theo:maintheorem} --- see the discussion in Section \ref{subsect:QNMs}, as well as the statements of Propositions \ref{prop:uppermodes}, \ref{prop:realQNF}, \ref{prop:exponentialbound}.

\end{inparaenum} \end{rem}

An important question is to what extent the passage from quasimodes to QNFs depends on the exact form of the Kerr--AdS metric. Observe already that axial symmetry of the Kerr--AdS metric plays an important role in the statement of Theorem \ref{theo:maintheorem}. This allows one to compensate for the fact that the Killing field $\partial_t$ is not timelike near the event horizon for $a \neq 0$; see Propositions \ref{prop:uppermodes}, \ref{prop:realQNF} below. For the full range of parameters $|a|^2 < 3/|\Lambda|$, these propositions apply to stationary, \emph{axisymmetric} perturbations of the metric (throughout, perturbations are assumed to be small). 
 
On the other hand observe that Proposition \ref{prop:exponentialbound}, the final ingredient in the proof of Theorem \ref{theo:maintheorem}, is always stable under stationary perturbations of the metric. The analysis is based on a general microlocal framework developed by Vasy \cite{vasy:2013}, which is highly robust --- see \cite[Section 2.7]{vasy:2013} for a precise discussion.
 \section{Proof of Theorem \ref{theo:maintheorem}}

The proof of Theorem \ref{theo:maintheorem} relies on three key results, Propositions \ref{prop:uppermodes}, \ref{prop:realQNF}, \ref{prop:exponentialbound}, stated in the next section. The proof of Proposition \ref{prop:exponentialbound} is delayed until Section \ref{sect:exponentialbound}, while Propositions \ref{prop:uppermodes}, \ref{prop:realQNF} are proved in Section \ref{sect:uppermodes} at the end of the paper.

As mentioned in the introduction, Theorem \ref{theo:maintheorem} was established for Schwarzschild--AdS metrics previously \cite{gannot:2014} (namely when the angular momentum $a$ vanishes). Compared with the rotating case considered here, several simplifications were available:
\begin{enumerate} \itemsep6pt
	\item Staticity and spherical symmetry of the metric allowed for a decomposition of the stationary wave operator $P(\lambda)$ (defined in Section \ref{subsect:QNMs} below) into a family of one-dimensional Euclidean Schr\"odinger operators $P_l - \lambda^2$, where the index $l$ corresponds to a fixed space of spherical harmonics.
	\item The Killing field $\partial_t$ is timelike outside the event horizon, so $P(\lambda)$ is elliptic in the exterior. This made it possible to realize each $P_l$ as a \emph{self-adjoint} operator, and therefore the analogue of Proposition \ref{prop:uppermodes} followed from standard facts about resolvents of self-adjoint operators away from the spectrum.
	
	\item Ellipticity combined with analyticity of the metric allows one to meromorphically continue each resolvent $(P_l-\lambda^2)^{-1}$ into the lower half-plane by the method of complex scaling \cite{sjostrand1991complex} (observe that the exact Kerr--AdS metric is also analytic, and while complex scaling has been successfully applied to some analytic rotating black hole metrics by Dyatlov \cite{dyatlov:2011:cmp}, that method is not very robust). Therefore QNFs for a fixed $l$ can be defined as poles of the continued resolvent. Proposition \ref{prop:realQNF}, namely the lack of nonzero poles on the real axis, becomes an elementary observation about ordinary differential equations.
	
	\item The exponential resolvent estimate of Proposition \ref{prop:exponentialbound} is well known in Euclidean scattering \cite{petkov2001semi,tang1998quasimodes}, hence could be applied directly.
\end{enumerate}
As will be clear from the proofs and discussions preceding them, each of these items becomes more involved in the rotating case. Primarily this is due to the lack of staticity and spherical symmetry of the metric, as well as the failure of ellipticity.  Even giving an effective definition QNFs is nontrivial --- see Theorem \ref{theo:QNF} below. However, once equipped with Propositions \ref{prop:uppermodes}, \ref{prop:realQNF}, \ref{prop:exponentialbound}, the proof of Theorem \ref{theo:maintheorem} is essentially the same as in the Schwarzschild--AdS setting (or more generally in the setting of \cite{tang1998quasimodes}), since it relies only on some abstract complex analysis.

\subsection{Kerr--AdS metric}

The Kerr--AdS metric is determined by three parameters $(\Lambda,M,a)$, where $\Lambda <0$ is the negative cosmological constant, $M > 0$ is the black hole mass, and $a \in \mathbb{R}$ is the angular momentum per unit mass. It is always possible to choose units such that $\Lambda = -3$, in which case the rotation speed is required to satisfy the regularity condition $|a| < 1$. Introduce the quantities 
\begin{gather*}
\Delta_r = \left(r^2+a^2\right)\left(1+ {r^2}\right) - 2M r,  \quad 
\Delta_\theta = 1- a^2\cos^2\theta, \\
\varrho^2 = r^2 + a^2 \cos^2 \theta.
\end{gather*}
When it exists, the largest positive root of $\Delta_r$ is denoted by $r_+$. The Kerr--AdS metric is given in Boyer--Lindquist coordinates by the expression
\begin{align*}
g  =& -\varrho^2 \left( \frac{d r^2}{\Delta_r} + \frac{d\theta^2}{\Delta_\theta} \right) - \frac{\Delta_\theta \sin^2 \theta}{\varrho^2(1-a^2)^2 }\left(a\, d t - (r^2 + a^2) \, d{ \phi}\right)^2 \\&+ \frac{\Delta_r}{\varrho^2(1-a^2)^2 }\left(d t - a\sin^2\theta \, d{\phi}\right)^2.
\end{align*}
Here $(\theta,\phi) \in (0,\pi) \times \mathbb{R}/(2\pi\mathbb{Z})$ are spherical coordinates on $\mathbb{S}^2$, while $t \in \mathbb{R}$ and $r \in (r_+,\infty)$. The dual metric $g^{-1}$ is given by
\begin{align*}
g^{-1} = & \ -\frac{\Delta_r}{\varrho^2} \partial_{r}^2 - \frac{\Delta_{\theta}}{\varrho^2} \partial_{\theta}^2 - \frac{(1-a^2)^2}{\varrho^2 \Delta_\theta \sin^2  \theta} \left( a \sin^2 \theta \partial_{{t}} + \partial_{{\phi}}\right)^2 \\&+ \frac{(1-a^2)^2}{\varrho^2\Delta_r}\left((r^2+a^2)\partial_{{t}} + a \partial_{{\phi}}\right)^2.
\end{align*}
The metric $g$ becomes singular at the event horizon $\{ r= r_+\}$, but this can be remedied by a change of variables. Set
 \begin{equation} \label{eq:kerrstar}
 t^\star = t + F_t(r); \quad \phi^\star = \phi + F_\phi(r),
 \end{equation}
 where $F_t,\, F_\phi$ are smooth functions on $(r_+,\infty)$ such that
 \begin{equation} \label{eq:changeofcoords}
 F'_t(r) = \frac{1-a^2}{\Delta_r}(r^2+a^2) + f_+(r), \quad  F'_\phi(r) = a\frac{1-a^2}{\Delta_r}.
 \end{equation}
 Here $f_+$ is a smooth function, and $F_t,\, F_\phi$ are chosen to vanish at infinity (in particular, $f_+$ must tend to zero as well). The dual metric in $(t^\star,r,\theta,\phi^\star)$ coordinates reads
 \begin{align} \label{eq:extendedmetric}
 \varrho^2 g^{-1} = &-\Delta_r 
 \left( \partial_r + f_+ \partial_{t^\star} \right)^2 - \Delta_\theta \partial_\theta^2 - 2(1-a^2) \left( \partial_r + f_+ \partial_{t^\star} \right)\left((r^2+a^2)\partial_{t^\star}  +a \partial_{\phi^\star} \right) \notag \\ &- \frac{(1-a^2)^2}{ \Delta_\theta \sin^2 \theta} \left( a \sin^2 \theta \partial_{t^\star} +  \partial_{{\phi^\star}}\right)^2.
 \end{align}
 Given $\delta \geq 0$, let $X_\delta = (r_+ -\delta,\infty) \times \mathbb{S}^2$ and $\mathcal{M}_\delta = X_\delta \times \mathbb{R}_{t^\star}$. Then \eqref{eq:extendedmetric} defines a Lorentzian metric, also denoted by $g$, on $\mathcal{M}_\delta$ for $\delta>0$ sufficiently small. There are two Killing vector fields, 
 \[
 T = \partial_{t^\star},\quad \Phi = \partial_{\phi^\star},
 \] 
 corresponding to stationarity and axisymmetry of Kerr--AdS.  Note that $T = \partial_t$ and $\Phi = \partial_\phi$ in the Boyer--Lindquist coordinates. The level set 
 \[
 \{ t^\star = 0\}\subseteq \mathcal{M}_\delta
 \] 
 is identified with $X_\delta$, and the function $f_+$ is chosen such that $X_\delta$ is spacelike; one explicit choice is $f_+(r) = (a^2-1)/(r^2+1)$.

 The spacetime $\mathcal{M}_\delta$ can be partially compactified by gluing in two boundary components 
 \[
 \mathcal{H}_\delta = \{ r= r_+-\delta\}, \quad \mathcal{I} = \{ r^{-1} = 0\}.
 \]
 Let $\OL{\mathcal{M}}_\delta = \mathcal{M}_\delta \cup \partial \mathcal{M}_\delta$, where $\partial\mathcal{M}_\delta = \mathcal{H}_\delta \cup \mathcal{I}$. Near $\mathcal{I}$, introduce the boundary defining function $s= r^{-1}$. From \eqref{eq:extendedmetric}, the conformal multiple $s^2 g$ has a smooth extension up to $\mathcal{I}$, and furthermore $s^{-2} g^{-1}(ds,ds)= -1$ on $\mathcal{I}$. There is a corresponding compactification at the level of time slices,
 \begin{equation} \label{eq:Xcompactification}
 \OL{X}_\delta = X_\delta \cup H_\delta \cup Y,
 \end{equation}
 where $H_\delta = \mathcal{H}_\delta \cap \{t^\star = 0\}$ and $Y = \mathcal{I} \cap \{t^\star = 0\}$.
 
 \begin{rem}
  If $f_+(r) = (a^2-1)/(r^2+1)$, then $\mathcal{I}$ intersects $\OL{X}_\delta$ orthogonally with respect to $s^2g$. This condition guarantees that the stationary Klein--Gordon operator $P(\lambda)$, defined in Section \ref{subsect:QNMs} below, is a Bessel operator in the sense of \cite[Section 1.5]{gannot:bessel} --- see \cite[Section 2]{gannot:bessel} for more details.
  \end{rem}

Finally, it must be assumed that the root of $\Delta_r$ at $r=r_+$ is simple, or equivalently $\Delta'_r(r_+) > 0$. This implies that the surface gravity 
\begin{equation} \label{eq:surfacegravity}
\varkappa = \frac{\Delta'_r(r_+)}{2(1-a^2)(r_+^2+a^2)}
\end{equation}
 of the event horizon is positive.

\subsection{Quasinormal modes} \label{subsect:QNMs}

Let $P(\lambda)$ denote the operator on $X_\delta$ obtained by replacing $D_{t^\star}$ with $-\lambda \in \mathbb{C}$ in the expression for $\varrho^2(\Box_g + \nu^2 - 9/4)$. Because the metric is stationary, $P(\lambda)$ is a well defined second order operator: if $u \in C^\infty(X_\delta)$ and $\lambda \in \mathbb{C}$, then
\[
P(\lambda)u = e^{i\lambda t^\star} \varrho^2 \left(\Box_g + \nu^2 - 9/4 \right) e^{-i\lambda t^\star}u.
\]
 Since $\Phi$ is also Killing, $P(\lambda)$ additionally preserves the space of distributions
\[
\mathcal{D}'_k(X_\delta) = \{ u \in \mathcal{D}'(X_\delta) : (\Phi-ik)u=0
\}.
\]
Define $\mathcal{L}^2(X_\delta)$ as the space of distributions $u$ for which
\[
\| u \|_{\mathcal L^2(X_\delta)} = \int_{X_\delta} |u|^2 \, r^{-1} \, dS_t < \infty,
\]
where $dS_t$ is the surface measure induced on $X_\delta$ by $g$. Motivated by the renormalization scheme introduced in \cite{warnick:2013:cmp}, the stationary energy space is defined with respect to the conjugated differential
\[
\tilde{d} u = r^{\nu-3/2} d \left( r^{3/2-\nu} u \right).
\]
If $h$ is the restriction of $-g$ to $TX_\delta$ (which is positive definite since $X_\delta$ is spacelike), then $\Sob^1(X_\delta)$ is defined as the space of distributions $u$ for which
\[
\| u \|_{\Sob^1(X_\delta)} = \int_{X_\delta} \left( |u|^2 + r^2\, h^{-1}\big(\tilde{d}u,\tilde{d}\bar u\big) \right)\, r^{-1} \, dS_t < \infty.
\]
The various powers of $r$ appearing as weights originate from natural energy identities \cite{holzegel2010massive,holzegel:2012:jhde,holzegel2013decay,warnick:2013:cmp,warnick:2015:cmp}, see also Section \ref{subsect:approximateinfinity} of this paper.

If $\nu\in (0,1)$, then boundary conditions must be imposed at the conformal boundary $Y$ to obtain the Fredholm property for $P(\lambda)$, recalling \eqref{eq:Xcompactification} for the definition of $Y$. Given $u \in \Sob^1(X_\delta)$ such that  $P(\lambda)u\in \mathcal{L}^2(X_\delta)$, the following boundary values are well defined in the Sobolev sense:
\begin{equation} \label{eq:rtraces}
\gamma_- u = \lim_{r\rightarrow \infty} r^{3/2-\nu}u, \quad \gamma_+ u = \lim_{r\rightarrow\infty} r^{2\nu+1}\partial_r (r^{3/2-\nu}u).
\end{equation}
Thus $\gamma_- u$ and $\gamma_+ u$ are analogues of Dirichlet and Neumann data for $u$. In fact, $P(\lambda)$ is elliptic at the boundary $Y$ in the sense of Bessel operators \cite[Section 1.5]{gannot:bessel}, so from the remark following \cite[Lemma 4.13]{gannot:bessel} and \cite[Proposition 3.6]{gannot:bessel}, one has the Sobolev regularity $\gamma_\pm u \in H^{\mp \nu}(Y)$.

Throughout this paper only self-adjoint Dirichlet or Robin boundary conditions are considered. The trace operator $B$ is therefore
\[
B = \gamma_- \text{ or } B= \gamma_+ + \beta \gamma_-,
\] 
where $\beta \in C^\infty(Y ;\mathbb{R})$ is a real valued function on the conformal boundary. Now define the domain
\begin{equation} \label{eq:Xspace}
\mathcal{X}(X_\delta) = \begin{cases} u \in \Sob^1(X_\delta): P(0)u \in \mathcal{L}^2(X_\delta)  &\text{ if $\nu \geq 1$,} \\[2pt]
u \in \Sob^1(X_\delta): P(0)u \in \mathcal{L}^2(X_\delta) \text{ and } Bu = 0  &\text{ if $\nu \in (0,1)$.}\end{cases}
\end{equation}
This is a Hilbert space for the norm $\| u \|_{\mathcal{X}(X_\delta)} = \| u \|_{\Sob^1(X_\delta)} + \|P(0)u\|_{\mathcal L^2(X_\delta)}$.

The following theorem was proved in \cite{gannot:2014:kerr} (and also in \cite{warnick:2015:cmp} under slightly more restrictive hypotheses).

\begin{theo} \label{theo:QNF}
For each $\nu > 0$ the operator $P(\lambda) : \mathcal{X}(X_\delta) \rightarrow \mathcal{L}^2(X_\delta)$ is Fredholm of index zero in the half-plane $\{\Im \lambda > -\tfrac{1}{2}\varkappa\}$, where the surface gravity $\varkappa > 0$ is given by \eqref{eq:surfacegravity}. Furthermore, 
\[
R(\lambda) = P(\lambda)^{-1} : \mathcal{L}^2(X_\delta) \rightarrow \mathcal{X}(X_\delta)
\] 
is a meromorphic family of operators in $\{ \Im \lambda > -\tfrac{1}{2}\varkappa\}$ which is holomorphic in any angular sector of the upper half-plane provided $|\lambda|$ is sufficiently large.
\end{theo} 
QNFs are defined as poles of the meromorphic family $R(\lambda)$. More information is needed about possible QNFs in the upper half-plane --- this is closely related to the boundedness of solutions to the Klein--Gordon equation. One remarkable property of rotating Kerr--AdS metrics is that for $|a| < r_+^2$ the vector field 
\[
K = T + \frac{a}{r_+^2+a^2} \Phi
\]
generating the Killing horizon $\mathcal{H}^+ =\mathcal{H}_0$ is in fact everywhere timelike on $\mathcal{M}_0$. The existence of such a vector field eliminates possible superradiant phenomena. For black holes satisfying the Hawking--Reall bound $|a| <r_+^2$, boundedness \cite{holzegel:2012wt} (and in fact logarithmic decay \cite{holzegel2013decay}) is known for solutions to the Klein--Gordon equation under Dirichlet boundary conditions. If the condition $|a| < r_+^2$ is violated, then it is possible to construct mode solutions $e^{-i\lambda t^\star}u$ which grow exponentially in time \cite{dold:2015cqa}, corresponding to QNFs in $\{\Im \lambda > 0\}$.

\begin{rem}
	Interestingly, even for Neumann boundary conditions boundedness has not been established for the expected range of black hole parameters $|a| < r_+^2$, see the conjecture in \cite[Section 5]{holzegel:2012wt}.
\end{rem}

For mode solutions $e^{-i\lambda t^\star}u$ many of the delicate issues involving lower order terms and boundary conditions are overcome in the high frequency limit. In fact, by working at a fixed axial mode it is not even necessary to restrict below the Hawking--Reall bound; of course this is only possible because of the axisymmetry of the Kerr--AdS metric. In that case the Robin function $\beta$ should satisfy $\Phi\beta = 0$ as well. Note that $R(\lambda)$ decomposes as a direct sum of operators 
\[
R(\lambda,k) : \mathcal{L}^2(X_\delta) \cap \mathcal{D}_k'(X_\delta) \rightarrow \mathcal{X}(X_\delta) \cap \mathcal{D}_k'(X_\delta),
\]
where $R(\lambda,k)$ is the restriction of $R(\lambda)$ to the closed subspace $\mathcal{L}^2(X_\delta) \cap \mathcal{D}_k'(X_\delta)$. In particular, $\lambda_0$ is a QNF if and only if there exists $k_0 \in \mathbb{Z}$ such that $\lambda_0$ is a pole of $R(\lambda,k_0)$. The following crucial proposition quantifies the absence of QNFs at a fixed axial mode in the upper half-plane at high frequencies.
\begin{prop} \label{prop:uppermodes}
Given $k \in \mathbb{Z}$ there exists $C_0>0$ such that if $\lambda \in \mathbb{R}\setminus [-C_0,C_0] + i(0,\infty)$, then  
	\begin{equation} \label{eq:uppermodes}
	\| u \|_{\mathcal{L}^2(X_0)} \leq \frac{C}{|\lambda|\Im \lambda} \|P(\lambda)u \|_{\mathcal{L}^2(X_0)}
	\end{equation}
	for each  $u \in \mathcal{X}(X_0) \cap \mathcal{D}'_k(X_0)$.
\end{prop}
Proposition \ref{prop:uppermodes} is stated for functions on the exterior time slice $X_0$ rather than the extended region $X_\delta$. On the other hand, \cite[Proposition 7.1]{gannot:2014:kerr} combined with Proposition \ref{prop:uppermodes} implies that $R(\lambda,k)$ has no poles in the region $\mathbb{R} \setminus [-C_0,C_0] + i(0,\infty)$. Moreover, if $\mathcal{R} : \mathcal{L}^2(X_\delta) \rightarrow \mathcal{L}^2(X_0)$ is the restriction operator, then \eqref{eq:uppermodes} implies
\begin{equation} \label{eq:uppermodes2}
\| \mathcal{R} R(\lambda,k)f \|_{\mathcal{L}^2(X_0)} \leq \frac{C}{|\lambda|\Im \lambda} \| f \|_{\mathcal{L}^2(X_\delta)}
\end{equation}
for each $f \in \mathcal{L}^2(X_\delta) \cap \mathcal{D}_k(X_\delta)$ and $\lambda \in \mathbb{R}\setminus [-C_0,C_0] + i(0,\infty)$. The constants $C_0, C>0$ in \eqref{eq:uppermodes2} a priori depend on $k \in \mathbb{Z}$.

\begin{rem} Proposition \ref{prop:uppermodes} is also valid without restricting to a fixed axial mode provided the Hawking--Reall bound holds (as will be evident from the proof). 
	\end{rem}

It is also important to know that there are no QNFs on the real axis. The proof also exploits axisymmetry of the Kerr--AdS metric, hence the Robin function should satisfy $\Phi\beta=0$ as well.

\begin{prop} \label{prop:realQNF}
		If $u \in \Sob^1(X_\delta) \cap \mathcal{D}'_k(X_\delta)$ solves $P(\lambda)u = 0$ for $\lambda \in \mathbb{R}$ satisfying $(r_+^2+a^2)\lambda \neq ak$, then $u=0$.
\end{prop}
\noindent If $k \in \mathbb{Z}$ is fixed, then poles of $R(\lambda,k)$ on the real axis are only possible for one exceptional value of $\lambda$; certainly $R(\lambda,k)$ has no real poles if $|\lambda|$ is sufficiently large.

The crucial final ingredient used to prove Theorem \ref{theo:maintheorem} is an exponential bound on $R(\lambda)$ in a strip away from suitable neighborhoods of the poles of $R(\lambda)$. At this stage it is convenient to introduce the semiclassical rescaling. Given a parameter $h > 0$, set
\[
P_h(z) = h^2 P(h^{-1}z), \quad R_h(z) = P_h(z)^{-1}.
\] 
For the next proposition the Robin function $\beta$ is not required to satisfy $\Phi \beta = 0$; in fact, no serious integrability properties of the metric are used. Fix compact intervals $[a,b] \subseteq (0,\infty),\ [C_-,C_+] \subseteq (-\varkappa/2,\infty)$, and define
\begin{equation} \label{eq:omega}
 \Omega(h) = [a, b ] + i h [C_-, C_+ ].
 \end{equation}
 Let $\{ z_j\}$ enumerate the (discrete) poles of $R_h(z)$ in $\{ \Im z > -\tfrac{1}{2}h\varkappa\}$; for each $h > 0$ and $\delta > 0$, only finitely many disks $B(z_j,\delta)$ intersect $\Omega(h)$.

\begin{prop} \label{prop:exponentialbound}
There exists $A > 0$ such that for any function $0 < S(h) = o(h)$ there holds the estimate
\begin{equation} \label{eq:expbound}
\|R_h(z) \|_{\mathcal{L}^2(X_\delta) \rightarrow \mathcal{L}^2(X_\delta)} < \exp \left(Ah^{-9}\log(1/S(h))\right), 
\end{equation}
provided
\[
z \in \Omega(h) \setminus \bigcup_j B(z_j,S(h))
\]
and $h>0$ is sufficiently small. 
\end{prop}

Proposition \ref{prop:exponentialbound} is proved in Section \ref{subsect:singular}. Finally, the quasimode construction of \cite{holzegel2014quasimodes} is reviewed:

\begin{theo} [{\cite[Theorem 1.2]{holzegel2014quasimodes}}] \label{theo:quasimodes}
	There exists a sequence ${\lambda}^\sharp_\ell \in \mathbb{R}$ and 
	\[
	u^\sharp_\ell \in \mathcal{X}(X_\delta) \cap C^\infty(X_\delta) \cap \mathcal{D}'_0(X_\delta) 
	\]
	with the following properties.
	\begin{enumerate} \itemsep6pt
		\item There exists $C > 0$ such that $\ell/C < \lambda^\sharp_\ell < C\ell$.
		\item $\supp u^\sharp_\ell \subseteq (r_1,\infty)$ for some $r_1 > r_+$ independent of $\ell$.
		\item The functions $u_\ell^\sharp$ are normalized exponentially accurate quasimodes in the sense that 
		\[
		\| u_\ell^\sharp \|_{\mathcal L^2(X_\delta)} = 1, \quad 
		\| P\big({\lambda}^\sharp_\ell\big)u_\ell^\sharp \|_{\mathcal{L}^2(X_\delta)} \leq e^{-\ell/D}
		\]
		for some $D > 0$. Furthermore, $\gamma_- u_\ell^\sharp =0$ if $\nu \in (0,1)$.
	\end{enumerate}
	\end{theo}  

As remarked in \cite[Footnote 8]{holzegel2014quasimodes}, the boundary condition $B = \gamma_-$ in \cite[Theorem 1.2]{holzegel2014quasimodes} could also be replaced by a Robin boundary condition of the form $B= \gamma_+ + \beta \gamma_-$, where $\beta \in \mathbb{R}$ is a real constant. The only difference is that the Hardy inequality used in \cite{holzegel2014quasimodes} now holds modulo a boundary term which is negligible in the semiclassical limit.

The proof of Theorem \ref{theo:maintheorem} is finished by the same arguments as in the work of Tang--Zworski \cite{tang1998quasimodes}, with refinements by Stefanov \cite{stefanov1999quasimodes,stefanov2005approximating}.

\begin{proof} [Proof of Theorem \ref{theo:maintheorem}]
Define a semiclassical parameter $h>0$ by $h^{-1} = \lambda^\sharp_\ell$ (so $h\rightarrow 0$ as $\ell\rightarrow\infty$), and then set $u(h) = u_\ell^\sharp$. Suppose that $R_h(z,0)$ was holomorphic in the rectangle 
\begin{equation} \label{eq:rectangle}
[1-2w(h)-S(h),1+2w(h)+S(h)] + i [-2Ah^{-9}\log(1/S(h))S(h),S(h)],
\end{equation}
where $A>0$ is provided by Proposition \ref{prop:exponentialbound}, and $w(h), \,S(h)$ are to be specified. Then
\[
F(z) = \mathcal{R}R_h(z,0) :   \mathcal{L}^2(X_\delta) \cap \mathcal{D}'_0(X_\delta) \rightarrow \mathcal{L}^2(X_0)\cap  \mathcal{D}'_0(X_0)
\]
is certainly holomorphic in the smaller rectangle
\[
\Sigma(h) = [1-2w(h),1+2w(h)] + i [-Ah^{-9}\log(1/S(h))S(h),S(h)],
\]
and satisfies the operator norm estimates
\[
\| F(z) \| < \begin{cases} C/\Im z &\text{ for $z \in {\Sigma}(h) \cap \{ \Im z >0\}$},\\ e^{Ah^{-9}\log(1/S(h))} &\text{ for $z \in {\Sigma}(h)$}, \end{cases}
\]
according to Propositions \ref{prop:uppermodes}, \ref{prop:exponentialbound}. Applying the semiclassical maximum principle \cite[Lemma 1]{stefanov2005approximating}, it follows that
\[
\| F(z) \| < e^3/S(h) \text{ for } z \in [1-w(h),1+w(h)],
\]
provided $w(h),\, S(h)$ are chosen so that $e^{-B/h} \leq S(h) < 1$ for some $B>0$ and
\begin{equation} \label{eq:w(h)}
18A  h^{-9}\log(1/h) \log(1/S(h))S(h) \leq w(h).
\end{equation}
Choose $B>0$ such that  $\|P_h(1)u(h)\|_{\mathcal{L}^2(X_\delta)}  < e^{-B/h}$. Then,
\[
1 = \| u(h) \|_{\mathcal{L}^2(X_\delta)} = \| F(1) P_h(1) u(h) \|_{\mathcal{L}^2(X_0)} < e^3 e^{-B/h}/S(h).
\]
This yields a contradiction by choosing $S(h) = 2e^3e^{-B/h}$ for example. Thus there must exist a pole in the rectangle \eqref{eq:rectangle}, which furthermore must lie in $\{ \Im z < 0 \}$ because of Propositions \ref{prop:uppermodes}, \ref{prop:realQNF}. Defining $w(h)$ by the left hand side of \eqref{eq:w(h)}, it follows that $R_h(z,0)$ has a pole $z(h)$ such that 
\[
| z(h) - 1 | < Ch^{-10}\log(1/h) e^{-B/h}, \quad 0 < -\Im z(h) < Ch^{-10}e^{-B/h}.
\]
Letting $\lambda_\ell  = \lambda_\ell^\sharp \cdot z\big(1/\lambda_\ell^\sharp\big)$ proves Theorem \ref{theo:maintheorem}.
\end{proof}

The rest of the paper is dedicated to proving Propositions \ref{prop:uppermodes}, \ref{prop:realQNF}, and \ref{prop:exponentialbound}.

\section{Exponential bounds on the resolvent} \label{sect:exponentialbound}

The first goal is to prove Proposition \ref{prop:exponentialbound}. For this, an approximate inverse is constructed for $P(\lambda)$ modulo an error of Schatten class. The approximate inverse is built up from local parametrices which invert $P(\lambda)$ near the event horizon and near the conformal boundary.  This is similar to the black-box approach of Sj\"ostrand--Zworski \cite{sjostrand1991complex}.

\subsection{Microlocal analysis of the stationary operator} \label{subsect:microlocal}

In order to construct a local parametrix in a neighborhood of the event horizon using methods of Vasy \cite{vasy:2013}, one needs to understand the Hamilton flow of the semiclassical principal symbol of $P_h(z)$ near its characteristic set. This necessitates a review of terminology common in semiclassical microlocal analysis; for a thorough exposition, see \cite[Appendix E]{zworski:resonances}.

It is convenient to view the (rescaled) flow on a compactified phase space $\OL{T^*}X_\delta$. The fibers of $\OL{T^*} X_\delta$ are obtained by gluing a sphere at infinity to the fibers of $T^*X_\delta$, so $\OL{T^*} X_\delta$ is a disk bundle whose interior is identified with $T^*X_\delta$ --- see \cite{melrose1994spectral,vasy:2013} as well as \cite[Appendix E]{zworski:resonances} for more details. If $| \cdot |$ is a smooth norm on the fibers of $T^*X_\delta$, then a function on $\OL{T^*}X_\delta$ is smooth in a neighborhood of $\partial \OL{T^*}X_\delta$ if it is smooth in the polar coordinates $(x,\rho = |\xi|^{-1}, \,\omega = |\xi|^{-1} \xi)$, where $(x,\xi) \in T^*X_\delta$.

Throughout, $\Im z$ will satisfy $|\Im z|< Ch$ for some $C>0$. In the semiclassical regime one may therefore assume that $z$ is in fact real. The semiclassical principal symbol $p = \sigma_h(P_h(z))$ is given by 
\[
p(x,\xi;z) = - g^{-1}(\xi \cdot dx - z \, dt^\star, \xi \cdot dx - z \, dt^\star),
\]
where $\xi\cdot dx$ is a typical covector on $X_\delta$.

Note that $\left<\xi \right>^{-2} p$ extends smoothly to a function on $\OL{T^*}X_\delta$ --- here $\left< \xi \right> = (1+|\xi|^2)^{1/2}$ with respect to the fixed norm $|\cdot|$ on $T^*X_\delta$. Furthermore, the rescaled Hamilton vector field $\left<\xi\right>^{-1} H_p$ extends to a smooth vector field on $\OL{T^*}X_\delta$ which is tangent to $\partial\OL{T^*}X_\delta$. The Hamilton flow of $p$ will refer to integral curves of $\left<\xi\right>^{-1}H_p$ in this compactified picture.

The characteristic set of $p$ is given by $\Sigma = \{ \left<\xi\right>^{-2}p = 0 \} \subseteq \OL{T^*}X_\delta$. Since $X_\delta \subseteq \mathcal{M}_\delta$ is spacelike, it is a general fact about Lorentzian metrics that for each $z\in \mathbb{R} \setminus 0$ the characteristic set is the union of two disjoint $z$-dependent sets $\Sigma = \Sigma_\pm$, where 
\[
\Sigma_\pm =  \Sigma \cap \{ \pm\left< \xi \right>^{-1} g(\xi\cdot dx - z \, dt^\star, dt^\star) <0 \}
\]
are the backwards and forwards light cones (intersected with a plane where the momentum dual to $dt^\star$ is $-z$) \cite[Section 3.2]{vasy:2013}. Each of these sets is invariant under the Hamilton flow. Finally, let 
\[
\widehat{\Sigma} = \Sigma \cap \partial\OL{T^*}X_\delta, \quad \widehat{\Sigma}_\pm = \Sigma_\pm \cap \partial\OL{T^*}X_\delta.
\] 
If $\partial\OL{T^*}X_\delta$ is identified with the quotient $S^*X_\delta = (T^*X_\delta\setminus 0)/\mathbb{R}_+$ by positive dilations in the fibers, then $\widehat{\Sigma}$ corresponds to the characteristic set of the homogeneous principal symbol of $P(\lambda)$ (as a non-semiclassical differential operator) within $S^*X$. In particular, $(x,\xi) \in \partial\OL{T^*}X_\delta\setminus \widehat{\Sigma}$ is equivalent to standard ellipticity of $P(\lambda)$ in the direction $(x,\xi)$.

If $x \in X_\delta$ and the vector field $T$ is timelike at $x$, then $p$ is elliptic near the fiber $\partial\OL{T^*_x} X_\delta$. Therefore the projection of $\widehat\Sigma$ onto the base space $X_\delta$ is contained within the ergoregion, namely the set where $T$ is not timelike. It is easily checked in Boyer--Lindquist coordinates that the ergoregion is described by the inequality $\Delta_r \leq a^2\Delta_\theta \sin^2\theta$.

Let  $\xi_r, \xi_{\theta}, \xi_{\phi^\star}$ denote the momenta conjugate to $r,\theta,\phi^\star$. Write $\Lambda_\pm \subseteq T^*X_\delta$ for the two components of the conormal bundle to $\{r = r_+\}$, 
\[
\Lambda_\pm = \{ r = r_+, \, \pm \xi_r > 0,\, \xi_\theta = \xi_{\phi^\star} = 0 \},
\]
and let $L_\pm$ denote the images in $\partial\overline{T^*}X_\delta$ of $\Lambda_\pm$ under the canonical projection $T^*X_\delta \setminus 0 \rightarrow \partial\overline{T^*}X_\delta$. These sets are invariant under the Hamilton flow of $p$, and $L_\pm \subseteq \widehat{\Sigma}_\pm$. The flow on $\overline{T^*}X_\delta$ is denoted by
\[
\varphi_t = \exp(t \left<\xi\right>^{-1} H_p )
\]
From the dynamical point of view, $L_+$ is a source and $L_-$ a sink for the Hamilton flow.

\begin{lem} [{\cite[Lemma 4.1]{gannot:2014:kerr}}] \label{lem:sourcesink}
	There exist neighborhoods $U_\pm$ of $L_\pm$ such that if $(x,\xi) \in U_\pm\setminus L_\pm$, then $\varphi_t(x,\xi) \rightarrow L_\pm$ as $\mp t \rightarrow \infty$ and $\varphi_{\pm T}(x,\xi) \notin U_\pm$ for some $T>0$.
\end{lem}

To analyze the flow more closely near $r=r_+$, observe that $H_p r$ evaluated at a point $(x,\xi)$ is given by
\[
H_p r = -2 \varrho^2 g^{-1}(\xi \cdot dx - z\,dt^\star, dr).
\]
Now $dr$ is null at $r=r_+$, which shows that $H_p r$ cannot vanish over $r=r_+$ at points $(x,\xi) \in \Sigma \cap T^*X_{\delta}$ with $z \in \mathbb{R}\setminus 0$. Indeed, if $H_p r=0$ at such a point, then the two null vectors $\xi\cdot dx - z\, dt^\star$ and $dr$ would be orthogonal and hence collinear; this is impossible since $z\neq 0$. Furthermore,
\[
g^{-1}(dr,dt^\star) = -\varrho^{-2}(1-a^2)(r^2+a^2) < 0
\]
at $r= r_+$, so $dr$ lies in the opposite light cone as $dt^\star$. Since $\Sigma_+$ is the backwards light cone and $\Sigma_-$ is the forward light cone, it follows that
\begin{equation} \label{eq:H_pr}
\pm H_p r < 0 \text{ on } \Sigma_\pm \cap T^*X_\delta \cap \{ r = r_+ \}
\end{equation}
for $z \in \mathbb{R}\setminus 0$.

\begin{lem} \label{lem:nontrapping}
	There exists $\delta > 0$ such that $\varphi_t$ satisfies the following conditions.
	\begin{enumerate} \itemsep6pt
		\item If $(x,\xi) \in \widehat{\Sigma}_\pm \setminus L_\pm$, then $\varphi_t(x,\xi) \rightarrow L_\pm$ as $\mp t \rightarrow \infty$, and $\varphi_{T}(x,\xi) \in \{ r \leq r_+ - \delta\}$
		for some $\pm T>0$.
		\item Suppose that $z \in \mathbb{R}\setminus 0$. If $(x,\xi) \in (\Sigma_\pm \setminus L_\pm) \cap \{|r-r_+| \leq \delta \}$, then either $\varphi_t(x,\xi) \rightarrow L_\pm$ as $\mp t\rightarrow \infty$ and 
		\[\varphi_{T}(x,\xi) \in \{ |r-r_+| \geq \delta\}
		\]
		for some $\pm T>0$, or
		\[
		\varphi_{\pm T_1}(x,\xi) \in \{ r \leq r_+ - \delta \}, \quad \varphi_{\mp T_2}(x,\xi) \in \{r \geq r_++\delta\}
		\] 
		for some $T_1,T_2 > 0$.

	\end{enumerate}
\end{lem}
\begin{proof}\begin{inparaenum}[1)]\item The first part is proved in Lemma \cite[Lemma 4.2]{gannot:2014:kerr}, which in turn follows from the same calculations as \cite[Section 6.3]{vasy:2013}. 
		
		\item The sets $U_\pm \supseteq L_\pm$ from Lemma \ref{lem:sourcesink} have the property that 
		\[
		K = \left(\Sigma \setminus (U_+ \cup U_-) \right) \cap \{ |r-r_+|  \leq \delta \}
		\]
		is a compact subset of $\OL{T^*}X_\delta$ for $\delta >0$ sufficiently small. By compactness and \eqref{eq:H_pr}, there exists $\varepsilon > 0$ such that $\mp H_p r > \varepsilon$ on $K \cap \Sigma_\pm$. 
		
		Consider first the case of $\Sigma_+$. If a flow line enters $U_+$, then it must tend to $L_+$ in the backward direction and permanently leave $U_+$ in the forward direction according to Lemma \ref{lem:sourcesink} (if it reentered $U_+$ in the forward direction, then the entire backward flow-out would have been contained in $U_+$). As the flow line leaves $U_+$, either $|r-r_+| \geq \delta$ or otherwise it enters the compact set $K$, at which point $H_p r < -\varepsilon$ has a definite sign. On the other hand, if a flow line never enters $U_+$, then it must escape $K$ at $r=r_+-\delta$ in the forward direction and $r=r_++\delta$ in the backward direction. The same argument applies to $\Sigma_-$ with the directions reversed.  \qedhere
	\end{inparaenum}
\end{proof}	

An idealization of the Hamilton flow near $r=r_+$ illustrating Lemma \ref{lem:nontrapping} is shown in Figure \ref{fig:dynamics} (also see Section \ref{subsect:approximateevent} for more on the notation used in the figure).

\subsection{Localizing the problem} \label{subsect:localizing}

As remarked in the introduction to this section, an approximate inverse for $P(\lambda)$ is contructed from local inverses, both near the event horizon and near infinity. Begin by decomposing 
\[
X_\delta =   \{ r_+ -\delta < r < R \} \cup \{ r > R/2 \},
\]
where $R> r_+$ will be fixed later. Dealing with the boundaries located at $r =R$ and $r = R/2$ is not convenient, and for this reason these cylinders are embedded in larger manifolds without boundaries at $r= R$ and $r=R/2$.

Let $X_+$ denote the cylinder $\{ r_+-\delta < r < 3R\}$ capped off with a 3-disk at $\{r = 3R\}$. The operator $P(\lambda)$, defined near $\{ r_+ -\delta < r < R \}$, must be extended to $X_+$ in a suitable way. To do this, define the auxiliary manifold ${\mathcal{M}}_+ = \mathbb{R}_{t^\star} \times X_+$. The goal is to extend the original metric $g$ to ${\mathcal{M}}_+$, and then consider the stationary Klein--Gordon operator with respect to the extended metric.

The extended metric will be chosen so that $\mathcal{M}_+$ is foliated by spacelike surfaces $\{ t^\star = \mathrm{constant}\}$. Such a metric is uniquely determined by its ADM decomposition:  begin with a smooth function $A > 0$, an arbitrary vector field $W$, and a Riemannian metric $h$, all defined initially on $X_+$. These objects are extended to $\mathcal{M}_+$  by requiring them to be independent of $t^\star$. 

The data $(A,W,h)$ uniquely determine a stationary Lorentzian metric $g_+$ on ${\mathcal{M}}_+$ by defining
\begin{gather} \label{eq:ADM}
g_+(T,T) = A^2 - h(W,W), \quad g_+(T,V) = -h(W,V), \notag\\ g_+(V,V) = -h(V,V),
\end{gather}
where $T = \partial_{t^\star}$ and $V$ is a vector in the tangent bundle of $\{t^\star = \mathrm{constant}\}$ viewed as a subbundle of $T\mathcal{M}_+$.

In standard terminology, $A$ is called the lapse function, and $W$ the shift vector associated to $g_+$. If $N_t$ denotes the unit normal to $\{t^\star = \mathrm{constant}\}$, then $A$ and $W$ can be uniquely recovered from $g_+$ via the formulas
\[
A = g_+(N_t,T) = g_+^{-1}(dt^\star,dt^\star)^{-1/2},\quad W = \partial_{t^\star} - A\cdot N_t.
\]
Similarly, $h$ is recovered as the induced metric on $\{t^\star = \mathrm{constant}\}$. The idea is to extend $g$ to ${\mathcal{M}}_+$ by extending the data $(A,W,h)$ originally defined near $\mathbb{R}_{t^\star} \times \{ r_+ -\delta < r < R \}$.

\begin{lem} \label{lem:adm}
	Let $R> r_+$ be such that $T$ is timelike for $g$ near $\{ r > R/10\}$. There exists a stationary Lorentzian metric $g_+$ on ${\mathcal{M}}_+$ such that 
	\begin{enumerate} \itemsep6pt
		\item $g_+ = g$ near $\mathbb{R}_{t^\star} \times \{ r_+ -\delta < r < R \}$,
		\item $dt^\star$ is everywhere timelike for $g_+$,
		\item $T$ is timelike for $g_+$ near $\mathbb{R}_{t^\star} \times \{ r > R/10 \}$.
	\end{enumerate}
\end{lem}
\begin{proof}
	Let $(A,W,h)$ denote the ADM data for the metric $g$, originally defined near $\mathbb{R}_{t^\star} \times \{ r_+ -\delta < r < R \}$. Fix any function $\widehat{A} > 0$ and Riemannian metric $\widehat{h}$ on ${X}_+$ such that
	\[
	\widehat{A} = A,\quad \widehat{h} = h
	\]
	near $\{r_+-\delta < r < 2R\}$. Choose a cutoff $\chi = \chi(r)$ with values $0\leq \chi \leq 1$, such that $\supp \chi \subseteq \{ r_+-\delta \leq r < 2R\}$ and $\chi = 1$ near $\{r_+-\delta \leq r < R\}$. Then, define $\widehat{W} = \chi W$. 
	
	The metric $g_+$ determined by $(\widehat{A},\widehat{W},\widehat{h})$ satisfies the requirements of the lemma; the only part which is not immediate is the last item. If $\chi > 0$, then $A = \widehat{A}$ and $h = \widehat{h}$, so 
	\[
	g_+(T,T) = A^2 - \chi^2 h(W,W) \geq A^2 - h(W,W) > 0.
	\]
	On the other hand, if $\chi = 0$, then $\widehat{W} = 0$ and so ${g}_+(T,T) = \widehat{A}^{\,2} > 0$.
 \end{proof}

Given Lemma \ref{lem:adm}, define the stationary operator
\[
P_+(\lambda) = e^{i\lambda t^\star}\varrho^2 \left(\Box_{g_+} + \nu^2-9/4\right)e^{-i\lambda t^\star}
\]	
on ${X}_+$, and similarly for the semiclassical version $P_{h,+}(z)$. The same argument can also be applied to the infinite part $\{ r > R/2\}$: in that case $X_\infty$ is obtained by capping $\{r \geq R/4\}$ at $r = R/4$. Then $g$ is extended to a stationary metric $g_\infty$ on ${\mathcal{M}}_\infty = \mathbb{R}_{t^\star} \times {X}_\infty$, and the analogue Lemma \ref{lem:adm} holds in such a way that $T$ is everywhere timelike for the extended metric. Thus it is possible to define an extended operator $P_\infty(\lambda)$ and $P_{h,\infty}(z)$.

\subsection{An approximate inverse near the event horizon} \label{subsect:approximateevent}
The manifold $X_+$ can be compactified by gluing in the boundary
\[
H_\delta = \{r = r_+ - \delta\}.
\]
If $L^2(X_+)$ is defined with respect to any positive density on the compact manifold with boundary $X_+ \cup H_\delta$, let ${H}^1(X_+)$ denote the space of distributions $u\in L^2(X_+)$ such that $du \in L^2(X_+)$ (with respect to any smooth norm on covectors). These distributions are extendible across $H_\delta$ in the sense of \cite[Appendix B.2]{hormanderIII:1985}, where this space is denoted $\OL{H}^1(X_+)$. In analogy with \eqref{eq:Xspace}, define
\[
\mathcal{X}_+ = \{ u \in {H}^1(X_+): P_+(0)u \in L^2(X_+)\},
\]
equipped with the norm $\|u\|_{\mathcal{X}_+} = \| u \|_{H^1(X_+)} + \| P_+(0)u\|_{L^2(X_+)}$.

\begin{prop} \label{prop:Rfredholm}
	The operator $P_+(\lambda) : \mathcal{X}_+ \rightarrow L^2(X_+)$ is Fredholm of index zero in the half-plane $\{\Im \lambda > -\tfrac{1}{2}\kappa\}$. Furthermore,\[
	R_+(\lambda):= P_+(\lambda)^{-1} : L^2(X_+) \rightarrow \mathcal{X}_+
	\]
	is a meromorphic family of operators in $\{\Im \lambda > -\tfrac{1}{2}\varkappa\}$ which is holomorphic in any angular sector of the upper half-plane, provided $|\lambda|$ is sufficiently large.
\end{prop}
\begin{proof}
	Let $p_+$ denote the semiclassical principal symbol of $P_{h,+}(z)$. Note that $p = p_+$ near the characteristic set of $p_+$ at fiber infinity (whose projection to $X_+$ is contained in a neighborhood of $\{ r < R/10\}$). Therefore the hypotheses at fiber infinity of \cite[Section 2.6]{vasy:2013} also hold for $p_+$, since they were verified for $p$ in \cite[Section 4]{gannot:2014:kerr}. That suffices to show that $P_+(0)$ is Fredholm on $\mathcal{X}_+$. 
	
	The invertibility statement follows from the ellipticity of $P_{h,+}(z)$ on compact subsets of phase space for $\Im z > 0$, which in turn is a corollary of the fact that $dt^\star$ is timelike for $g_+$ --- see \cite[Sections 3.2, 7]{vasy:2013}.
\end{proof}

A more refined invertibility statement for $R_+(\lambda)$ will be proved at the end of the section. Before doing this, one must also consider a nontrapping model where $P_{h,+}(z)$ is modified by a complex absorbing operator. 

The relevant class of semiclassical pseudodifferential operators $\Psi^{\mathrm{comp}}_{h}(X_+)$ are compactly supported and compactly microlocalized (the superscript $\mathrm{comp}$ refers to compact microlocalization). As usual, compact support means that the Schwartz kernel of any such operator is supported in a compact subset of $X_+ \times X_+$, while compactly microlocalized means that the semiclassical wavefront set of any such operator is a compact subset of the interior $T^*X_+ \subset \OL{T^*} X_+$. For more about semiclassical microlocal analysis on a noncompact manifold, see \cite[Appendix E]{zworski:resonances}. 

For the purpose of this paper it suffices to observe that any compactly supported function $a \in C_c^\infty(T^*X_+)$ can be quantized to obtain a compactly supported operator $A = \mathrm{Op}_h(a) \in \Psi^{\mathrm{comp}}_{h}(X_+)$ such that $\WF_h(A) \subseteq \supp a$.

In the next lemma, $\Sigma_\pm$ will denote the two components of the characteristic set $\Sigma$ of $p_+$ as in Section \ref{subsect:microlocal} (replacing $p$ with $p_+$). This separation is possible since $dt^\star$ is timelike for $g_+$. Furthermore $\varphi_t$ will denote the Hamilton flow of $p_+$.

\begin{lem} \label{lem:Qpm}
	Fix $[a,b] \subseteq (0,\infty)$. Then there exist compactly supported $z$-independent $Q_\pm \in \Psi_{h}^{\mathrm{comp}}(X_+)$ such that if $q_\pm = \sigma_h(Q_\pm)$, then 
	\begin{enumerate} \itemsep6pt
		\item $\mathrm{WF}_h(Q_+) \cap \mathrm{WF}_h(Q_-) = \emptyset$ and $\pm q_\pm \geq 0$,
		\item If $(x,\xi) \in \Sigma_\pm\setminus L_\pm$,  then either $\varphi_t \rightarrow L_\pm$ as $\mp t \rightarrow \infty$, or $\varphi_T(x,\xi) \in \{q_\pm \neq 0 \}$ for some $\mp T\geq0$, uniformly in $z \in [a,b]$.
	\end{enumerate} 
\end{lem}	
\begin{proof}
	According to Lemma \ref{lem:nontrapping}, there exist compact sets 
	\[
	K_\pm \subseteq \Sigma_\pm \cap \{ r \geq r_+ + \delta \} \cap T^*X_+
	\]
	such that if $(x,\xi) \in \Sigma_\pm \setminus L_\pm$ then either $\varphi_t(x,\xi) \rightarrow L_\pm$ as $\mp t\rightarrow\infty$, or $\varphi_T(x,\xi) \in K_\pm$ for some $\mp T \geq 0$. It then suffices to quantize functions $\pm q_\pm \geq 0$ which satisfy $\pm q_\pm > 0$ near $K_\pm$ and have compact support within $ T^*X_+$. This may be done uniformly for $z$ in any compact subset of $\mathbb{R} \setminus 0$.
\end{proof}

{\setlength\textfloatsep{0pt}
	\begin{figure}[t] \includegraphics[width=13cm]{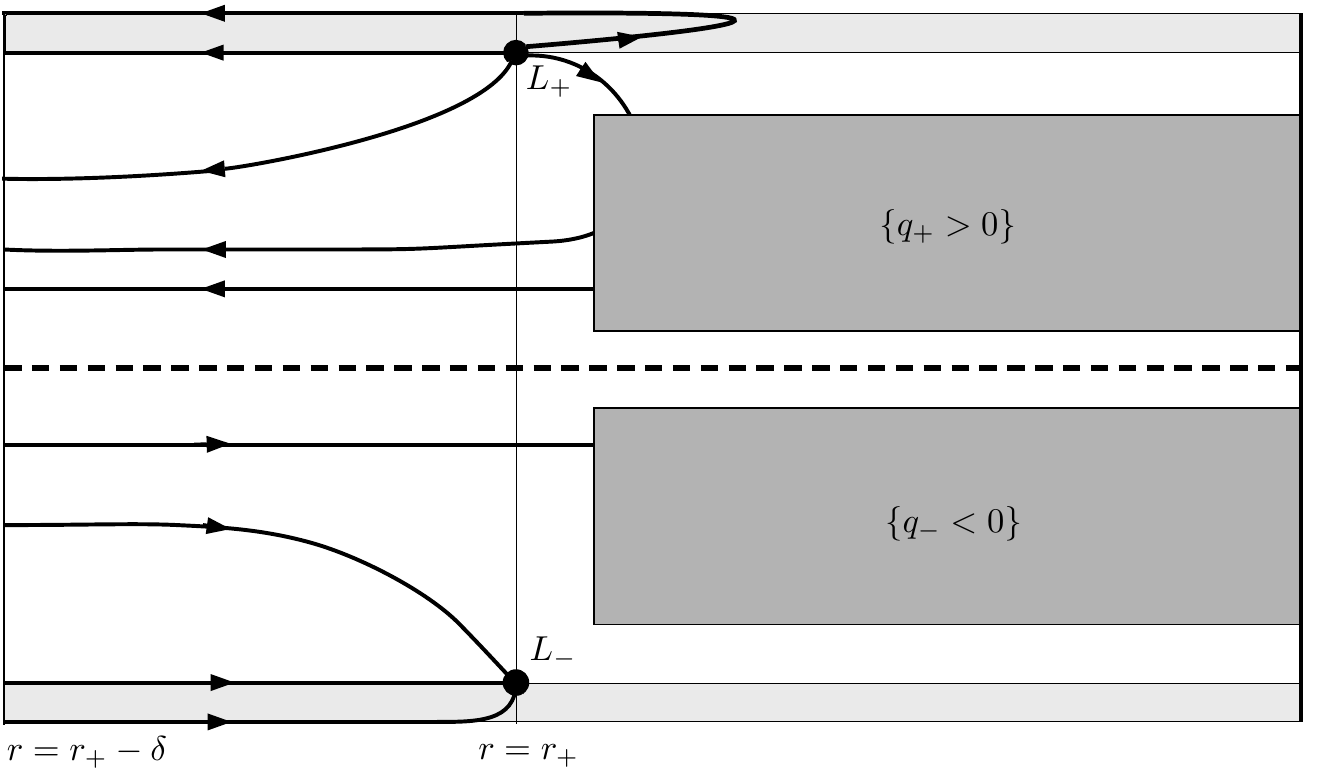} \label{fig:dynamics}
		\centering
		\caption{A schematic plot of the $\left<\xi\right>^{-1}H_{p_+}$ flow on $\overline{T^*}X_+$. The light grey shaded regions represent the boundary $\partial\overline{T^*}X_+$, and the horizontal dashed line is the intersection of $\{\left<\xi\right>^{-1} g(\xi\cdot dx - z\,dt^*,dt^\star)=0\}$ with the interior $T^*X_+$; thus $\Sigma_\pm$ are contained above/below the horizontal dashed line. The sources/sinks $L_\pm$ are represented by the bold dots, lying over $\{r =r_+\}$. Some flow lines of $\left<\xi\right>^{-1}H_{p_+}$ within $\Sigma$, and their directions, are drawn. Because $p_+$ is not elliptic, flow lines in $\widehat{\Sigma}$ can enter $\{r> r_+\}$, over the ergoregion; an example is drawn within $\widehat{\Sigma}_+$. The elliptic sets of $Q_\pm$ are also shown. Any flow line starting at a point in $\Sigma_+\setminus L_+$ either tends to $L_+$ or enters $\{q_+ \neq 0\}$ in the backward direction (see Lemmas \ref{lem:nontrapping}, \ref{lem:Qpm} for precise statements); the same is true for flow lines starting at points in $\Sigma_-\setminus L_-$, but with the directions reversed.}\end{figure} }

Fix an interval $[a,b] \subseteq (0,\infty)$ and define $Q = Q_+ + Q_-$, where $Q_\pm$ are given by Lemma \ref{lem:Qpm}. Then the modified operator $P_{h,+}(z) -iQ$
satisfies the nontrapping condition of \cite[Definition 2.12]{vasy:2013}, and hence \cite[Theorem 2.14]{vasy:2013} is valid:

\begin{prop} \label{prop:nontrapping}
	Fix $[a,b] \subseteq (0,\infty)$ and $[C_-, C_+] \subseteq (-\varkappa/2,\infty)$. Then there are $C>0$ and $h_0 > 0$ such that
	\[
	\mathcal{R}_{h,+}(z) := (P_{h,+}(z) -iQ)^{-1} : L^2(X_+) \rightarrow \mathcal{X}_+
	\]
	exists and satisfies the non-trapping bounds
	\[
	\| \mathcal{R}_{h,+}(z) u \|_{L^2(X_+) \rightarrow {H}^1_h(X_+)} \leq Ch^{-1}
	\]
	for $z \in [a,b] + ih[C_-,C_+]$ and $h \in (0,h_0)$. 
\end{prop}

\begin{rem} The semiclassical Sobolev space ${H}^1_h(X_+)$ appearing in Proposition \ref{prop:nontrapping} is ${H}^1(X_+)$ as a set, but equipped with the $h$-dependent norm 
	\[
	\| u \|_{{H}^1_h(X_+)} = \| u \|_{L^2(X_+)} + h\| du \|_{L^2(X_+)}.
	\]
	\end{rem}

The next task is to prove that $R_{h,+}(z)$ exists for $z\in [a,b] + i[M h,\infty)$ and $M > 0$ sufficiently large, with a corresponding estimate. This does not follow from Proposition \ref{prop:Rfredholm} since there one must take $\Im z > 0$ independent of $h$. The simplest way to prove this result is by energy estimates. For the next lemma it suffices to cite \cite[Lemma 4.6]{gannot:2014:kerr}, but see also Section \ref{subsect:approximateinfinity} where the same argument is adapted to the more complicated geometry near $Y$.

\begin{lem} \label{lem:invertible1}
	There exists $M>0$ and $C>0$ such that
	\[
	|\lambda| \| u \|_{L^2(X_+)} + \| du \|_{L^2(X_+)} \leq \frac{C}{\Im \lambda} \| P_+(\lambda) u \|_{L^2(X_+)}
	\]
	for each $\lambda \in \mathbb{R} + i[M,\infty)$ and  $u \in \mathcal{X}_+$.
\end{lem}
\begin{proof}
	This result is essentially proved in \cite[Lemma 4.6]{gannot:2014:kerr}, the only difference being that there $u$ was required to have support in $\{ r_+-\delta \leq r <r_1\}$ for some $r_1 > r_+$. Here, that condition is replaced by the compactness of $X_+ \cup H_\delta$. Although \cite[Lemma 4.6]{gannot:2014:kerr} applies to $u \in C^2(X_+ \cup H_\delta)$, the latter space is in fact dense in $\mathcal{X}_+$  --- see \cite[Lemma E.4.2]{zworski:resonances} or \cite[Section 2.6]{vasy:2013} for example.
\end{proof}

Lemma \ref{lem:invertible1} implies that $P_+(\lambda)$ is injective for $\lambda \in \mathbb{R} + i[M,\infty)$. Thus $P_+(\lambda)$ is invertible for such $\lambda$ since it is of index zero according to Proposition \ref{prop:Rfredholm}. Furthermore, after applying the semiclassical rescaling, there exists $C>0$ and $M>0$ such that
\begin{equation} \label{eq:horizoninvertible}
\| R_{h,+}(z) \|_{L^2(X_+)\rightarrow {H}^1_h(X_+)} \leq \frac{C}{\Im z}
\end{equation}
for $z \in \mathbb{R} + i[M h,\infty)$.

\subsection{An approximate inverse near infinity} \label{subsect:approximateinfinity}
The next step is to prove an analogue of Lemma \ref{lem:invertible1} for the operator $P_\infty(\lambda)$ defined in Section \ref{subsect:localizing}. This will now involve boundary contributions from the conformally timelike boundary $\mathcal{I}$. The function $s$ is extended as a positive function to all of $\mathcal{M}_\infty$ such that $Ts = 0$, and $r$ is extended as well via the formula $r = s^{-1}$. If $\varepsilon > 0$ is sufficiently small, then $\{s < \varepsilon\}$ defines a neighborhood of $\mathcal{I}$ in $\mathcal{M}_\infty$. Below, the following notation will be used:
\begin{itemize} \itemsep6pt
	\item $dS_t$ is the measure induced on $X_\infty$ and $A = g^{-1}_\infty(dt^\star,dt^\star)^{-1/2}$ is the lapse function for $g_\infty$,
	\item $d\mathcal{K}_\varepsilon$ is the induced measure on $X_\infty \cap \{s = \varepsilon\}$ and $A_{\varepsilon} = k_\varepsilon^{-1}(dt^\star,dt^\star)^{-1/2}$, where $k_\varepsilon$ is the induced Lorentzian metric on $\{ s= \varepsilon\}$.
	\item $N_t$ is the unit normal to $X_\infty$ pointing in the direction of increasing $t^\star$, and $N_s$ is the unit normal to $X_\infty \cap \{s = \varepsilon\}$ pointing in the direction of decreasing $s$ (the dependence of the latter on $\varepsilon$ is suppressed for convenience).
\end{itemize}
If $V$ is a $C^1$ vector field on $\mathcal{M}_\infty \cup \mathscr{I}$, then from the divergence theorem,
\begin{multline} \label{eq:divergence3}
\partial_{t^\star}  \int_{X_\infty \cap \{s \geq \varepsilon\}} g_\infty(V,N_t) \, d{S}_t - \int_{X_\infty \cap \{ s = \varepsilon \}} g_\infty(V,N_s) A_{\varepsilon} \, d\mathcal{K}_\varepsilon  \\ = \int_{X_\infty \cap \{ s \geq \varepsilon\}} \left(\mathrm{div}_{g_\infty} V\right) A\, dS_t,
\end{multline}
see \cite[Equation 4.3]{gannot:2014:kerr} as well as \cite[Lemma 3.1]{warnick:2015:cmp}. 
The data associated with the foliation by surfaces of constant $t^\star$ have conformal analogues: if $h_\infty$ is the induced metric on $X_\infty$, then
\[
A = s^{-1} \bar A, \quad h_\infty = s^{-2} \, \bar h_\infty,
\]
where $\bar A,\, \bar h_\infty$ are smooth up to $\mathscr{I}$. Similarly, define $\bar{A}_\varepsilon = sA_\varepsilon$; because $X_\infty$ and $\mathcal{I}$ intersect orthogonally with respect to $s^2g$,
\[
\bar A_\varepsilon \rightarrow \bar A|_{s=0}
\]
as $\varepsilon \rightarrow 0$. There are also conformally related measures on $X_\infty$ and $X_\infty \cap \{s = \varepsilon\}$ satisfying
\[
dS_t = s^{-3}d\bar{S}_t, \quad d\mathcal{K}_\varepsilon = s^{-2} d \bar{\mathcal{K}}_\varepsilon.
\]
In addition $N_t = s \bar N_t$ and $N_s = s \bar N_s$ where $\bar {N}_t, \bar N_s$ are the corresponding unit normals for $\bar{g}_\infty$, smooth up to $\mathcal{I}$. The goal is to eventually let $\varepsilon \rightarrow 0$ in \eqref{eq:divergence3}.

Written in terms of $s$, the traces $\gamma_\pm$ given by \eqref{eq:rtraces} have the form
\[
\gamma_- v = s^{\nu - 3/2} v|_{\mathcal{I}}, \quad \gamma_+ v = -s^{1-2\nu}\partial_s(s^{\nu-3/2}v)|_{\mathcal{I}}.
\]
Although $\gamma_\pm$ can be given weak formulations, for the energy estimates it is more useful to work with a space of smooth functions on which $\gamma_\pm$ are defined in the classical sense. Given $\nu \in (0,1)$, let $\mathcal{F}_\nu(\mathcal{M}_\infty)$ denote the space of all $v \in C^\infty(\mathcal{M}_\infty)$ admitting an expansion
\begin{equation} \label{eq:Fnu}
v(s,y) = s^{3/2+\nu}v_+(s^2,y) + s^{3/2-\nu}v_-(s^2,y)
\end{equation}
near $\mathcal{I}$, where $(s,y) \in [0,\varepsilon)\times \mathcal{I}$, and $v_\pm$ are smooth up to $\mathcal{I}$. If $v \in \mathcal{F}_\nu(\mathcal{M}_\infty)$ is given by \eqref{eq:Fnu}, then
\[
\gamma_- v(\cdot) = v_-(0,\cdot), \quad \gamma_+v (\cdot) = (-2\nu)v_+(0,\cdot).
\] 
If $\nu \geq 1$, say that $u \in \mathcal{F}_\nu(\mathcal{M}_\infty)$ if there exists $\delta >0$ such that $\supp u \subseteq \{ s > \delta\}$. The space $\mathcal{F}_\nu(X_\infty)$ is defined in the same way, simply replacing $\mathcal{M}_\infty$ with $X_\infty$.

For general boundary conditions (of the type considered in Section \ref{subsect:QNMs}), the boundary contribution on $X_\infty \cap \{s = \varepsilon\}$ arising from the usual stress-energy tensor will diverge as $\varepsilon \rightarrow 0$.  This can be remedied by introducing a ``twisted'' stress-energy tensor as in \cite{holzegel:2012wt,warnick:2013:cmp,warnick:2015:cmp} --- the reader is referred to these works for a more complete point of view. For this, fix a smooth function $q > 0$. Given a vector field $Y$ on $\mathcal{M}_\infty$, define the operator $\widetilde{Y}$ by
\[
\widetilde{Y}v = q \,Y (q^{-1} v),
\]
as well as the covector $\tilde{d} v = q\, d(q^{-1} v)$. The twisted stress-energy tensor $\widetilde{\mathbb{T}} = \widetilde{\mathbb{T}}[v]$ is defined by
\begin{align} \label{eq:twistedstressenergy}
\widetilde{\mathbb{T}}(Y,Z) &= \left[\Re \left(\widetilde{Y}v \cdot \widetilde{Z}\bar{v}\right) - \tfrac{1}{2}g_\infty(Y,Z)\,g_\infty^{-1}(\tilde{d}v,\tilde{d}\bar{v})\right] \notag \\ &+ \tfrac{1}{2}g_\infty(Y,Z)( Q + \nu^2 - 9/4) |v|^2 ,
\end{align}
where $Q = q^{-1}\Box_{g_\infty}(q)$ is a scalar potential. The term in square brackets is positive definite in $\tilde{d}v$ provided $Y,Z$ are timelike in the same lightcone. The twisting function $q$ is chosen so that 
\[
Q  + \nu^2 - 9/4 = \mathcal{O}(s^{2}).
\]
If $Y,Z$ are smooth up to $\mathscr{I}$, this guarantees that $g_\infty(Y,Z)(Q+\nu^2 -9/4)$ is also smooth up to $\mathscr{I}$. The simplest choice of $q$ with this property is $q = s^{3/2-\nu}$.

\begin{rem} Since multiplication by $Q + \nu^2 - 9/4$ is a zeroth order operator, the precise sign properties of $Q$ will not be important in the high frequency regime. More refined choices of $q$ leading to positive $Q$ are discussed at length in \cite{holzegel:2012wt}.
\end{rem}

Next, let $\widetilde{\mathbb{J}}^Y = \widetilde{\mathbb{J}}^Y[v]$ denote the unique vector field such that $g_\infty(\widetilde{\mathbb{J}}^Y,Z) = \widetilde{\mathbb{T}}(Y,Z)$.

\begin{lem} \label{lem:twisteddivergence}
Suppose that $Y$ is Killing for $g_\infty$ and $Yq = 0$. 	If $v \in \mathcal{F}_\nu(\mathcal{M}_\infty)$ and $F =(\Box_{g_\infty} + \nu^2 - 9/4)v$, then
	\[
	\mathrm{div}_{g_\infty} \, \widetilde{\mathbb{J}}^Y = \Re \left( F \cdot Y\bar{v}  \right)
	\]
\end{lem}
\begin{proof}
	Since $Y$ is Killing, the condition $Yq = 0$ also implies $YQ=0$, and then the result follows from a direct calculation \cite[Lemma 2.5]{warnick:2015:cmp}.
\end{proof}

If $q = s^{3/2-\nu}$, then $Tq = 0$ and hence Lemma \ref{lem:twisteddivergence} is valid with $Y=T$. Now apply \eqref{eq:divergence3} to the vector field $\widetilde{\mathbb{J}}^{T}$, where $v \in \mathcal{F}_\nu(\mathcal{M}_\infty)$. Consider the integral over $X_\infty \cap \{s \geq \varepsilon\}$, which can be written as
\[
\int_{X_\infty \cap \{s \geq \varepsilon\}} \widetilde{\mathbb{T}}(T,\bar{N}_t)\, s^{-2} \, d\bar{S}_t.
\]
Checking the various powers of $s$, this integral has a limit as $\varepsilon \rightarrow 0$ for $v \in \mathcal{F}_\nu(\mathcal{M}_\infty)$. This also motivates the following spaces: let $\mathcal{L}^2(X_\infty)$ denote the space of distributions for which
\[
\| u \|_{\mathcal{L}^2(X_\infty)} = \int_{X_\infty} |u|^2 \, s \,dS_t < \infty,
\]
and let $\Sob^1(X_\infty)$ denote the space of distributions for which
\[
\| u \|_{\Sob^1(X_\infty)} = \int_{X_\infty}\left( |u|^2 + s^{-2} h_{\infty}^{-1}\left(\tilde du, \tilde{d}\bar u\right) \right) s \,dS_t < \infty.
\]
Compare these spaces with those defined in Section \ref{subsect:QNMs}. 

Next, consider the integral in \eqref{eq:divergence3} over $X_\infty \cap \{s = \varepsilon\}$. This is only relevant in the case $\nu \in (0,1)$, since if $\nu \geq 1$ then the integral automatically vanishes for $v \in \mathcal{F}_\nu(\mathcal{M}_\infty)$ and $\varepsilon >0 $ sufficiently small. Since $T$ and $N_s$ are orthogonal and $Tq = 0$, this reduces to 
\begin{equation} \label{eq:boundaryintegral}
\int_{X_\infty \cap \{s=\varepsilon\}} \Re \left( Tv \cdot \widetilde{N}_s  \bar{v} \right) s^{-3} \bar{A}_{\varepsilon}  \, d\bar{\mathcal{K}}_\varepsilon.
\end{equation}
Furthermore, $N_s = -s\partial_s + \mathcal{O}(s^3)$.  Write
\[
s^{-3} \,Tv \cdot\widetilde{N}_s \bar{v} =\left( s^{\nu-3/2}Tv \right) \left( s^{-3/2-\nu}\widetilde{N}_s \bar{v} \right),
\]
and notice that this tends to $\gamma_- T v \cdot \gamma_+ \OL{v}$ as $\varepsilon\rightarrow 0$ for $v \in \mathcal{F}_\nu(X_\infty)$. Taking $\varepsilon \rightarrow 0$ in \eqref{eq:divergence3}, one therefore has the identity
\begin{multline} \label{eq:twisteddivergence}
\partial_{t^\star}\int_{X_\infty} \widetilde{\mathbb{T}}(T,\bar {N}_t) \,r^{-1{}}d{S}_t - \int_{Y} \Re \left(\gamma_- T v\cdot \gamma_+\bar{v} \right) \bar{A}\, d\bar{\mathcal K}_0  = \int_{X_\infty} \Re \left(   F \cdot T\bar {v} \right) A \,d{S}_t.
\end{multline}
Using \eqref{eq:twisteddivergence}, it is now straightforward to prove the analogue of Lemma \ref{lem:invertible1}. In the following, either $B = \gamma_-$ or $B = \gamma_+ + \beta \gamma_-$, where $\beta \in C^\infty(\mathscr{I};\mathbb{R})$ satisfies $T \beta = 0$.

\begin{lem} \label{lem:twistedinvertible}  There exists $C_0>0$ and $C>0$ such that
	\[
	|\lambda|\| u \|_{\mathcal{L}^2(X_\infty)} + \| u \|_{\Sob^1(X_\infty)} \leq \frac{C}{\Im \lambda} \| P_\infty(\lambda)u \|_{ \mathcal{L}^2(X_\infty)}
	\]
	for each $\lambda \in \mathbb{R} \setminus [-C_0,C_0] + i(0,\infty)$ and $u \in \mathcal{F}_\nu(X_\infty)$, provided $Bu = 0$ when $\nu \in (0,1)$.
\end{lem}
\begin{proof}
	The proof is close to that of \cite[Lemma 4.6]{gannot:2014:kerr}. Apply \eqref{eq:twisteddivergence} to a function $v = e^{-i\lambda t^\star}u$, where $u \in \mathcal{F}_\nu(X_\infty)$. Since $T$ and $\bar{N}_t$ are timelike in the same lightcone and are smooth up to $\mathcal{I}$, the first integral in \eqref{eq:twisteddivergence} controls
	\[
	\Im \lambda \left( 	|\lambda|^2\| u \|^2_{\mathcal{L}^2(X_\infty)} + \| u \|^2_{\Sob^1(X_\infty)} \right)
	\]
	for $\Im \lambda > 0$ and $|\lambda|$ sufficiently large. Let $f = P_\infty(\lambda)u$, and write the integral on the right hand side of \eqref{eq:twisteddivergence} in terms of the conformal measure; the integrand is therefore
	\[
	-s^{-4} \varrho^{-2}  \bar{A} \Im \left( \OL{\lambda u} \cdot f \right).
	\]
	From the Cauchy--Schwarz inequality $2ab \leq \delta^{-1}a^2 +\delta b^2$, this quantity is bounded by
	\begin{equation} \label{eq:petertopaul}
	-s^{-4} \varrho^{-2} \bar{A} \Im \left( \OL {\lambda u} \cdot f \right) \leq \frac{s^{-4} \varrho^{-2}}{2\delta\Im\lambda}\cdot |f|^{\,2} + \frac{\delta s^{-4} \varrho^{-2} \Im \lambda\, |\lambda|^2}{2} \bar{A} \cdot |u|^2.
	\end{equation}
	Now integrate \eqref{eq:petertopaul} over $X_\infty$ with respect to $d\bar{S}_t$, recalling that $\varrho^{-2} \sim s^2$. The integral of the first term on the right hand side of \eqref{eq:petertopaul} is bounded by a constant multiple of $(\Im\lambda)^{-1}\| f \|^2_{\mathcal{L}^2(X_\infty)}$, while the integral of the second term can be absorbed into the left hand side for $\delta>0$ sufficiently small.
	
	It remains to handle the integral over $Y$. If $u$ satisfies Dirichlet boundary condition (which recall is automatic for $\nu \geq 1$), then this term vanishes. Otherwise, if $\nu \in (0,1)$ and $B = \gamma_+ +\beta \gamma_-$, then the integrand becomes (after accounting for the minus sign in \eqref{eq:twisteddivergence})
	\begin{equation} \label{eq:boundarycontribution}
	 2 \Im \lambda \int_{Y} \beta \cdot |\gamma_- u|^2  \,\bar{A} \,d\bar{\mathcal{K}}_0.
	\end{equation}
	If $\beta$ is nonnegative this term can be dropped. Otherwise, for each $\delta > 0$ there exists $C_\delta>0$ such that
	\[
	\int_{Y} |\gamma_- u|^2  \,d\bar{\mathcal{K}}_0 \leq \delta\| u \|^2_{\Sob^1(X_\infty)} + C_\delta \| u \|^2_{\mathcal{L}^2(X_\infty)}.
	\]
	This is proved directly in \cite[Lemma B.1]{warnick:2015:cmp}, and also follows from \cite[Lemma 3.2]{gannot:bessel}. Hence the boundary term can always be absorbed by the left hand side for any $\Im \lambda > 0$ and $|\lambda|$ sufficiently large.
	\end{proof}

Now define the space
\begin{equation*}
\mathcal{X}_\infty = \begin{cases}  u \in \Sob^1(X_\infty): P_\infty(0)u \in \mathcal L^2(X_\infty)  &\text{ if $\nu \geq 1$,} \\[2pt]
 u \in \Sob^1(X_\infty): P_\infty(0)u \in \mathcal L^2(X_\infty) \text{ and } Bu = 0  &\text{ if $\nu \in (0,1)$,}\end{cases}
\end{equation*}
equipped with the graph norm --- compare to Section \ref{subsect:QNMs}. Note that $P_\infty(\lambda)$ is elliptic on $X_\infty$ in the sense of standard microlocal analysis, where $X_\infty$ is viewed as a noncompact manifold without boundary. Furthermore, $P_\infty(\lambda)$ is elliptic at $Y$ as a Bessel operator. By elliptic regularity for Bessel operators \cite[Theorem 1 and Lemma 4.13]{gannot:bessel} it is therefore possible to show that all computations in Lemma \ref{lem:twistedinvertible} for $u \in \mathcal{F}_\nu(X_\infty)$ are also valid for $u \in \mathcal{X}_\infty$. In particular,
\begin{equation} \label{eq:infinityestimate}
	|\lambda|\| u \|_{\mathcal{L}^2(X_\infty)} + \| u \|_{\Sob^1(X_\infty)} \leq \frac{C}{\Im \lambda} \| P_\infty(\lambda)u \|_{ \mathcal{L}^2(X_\infty)}
\end{equation}
	for each $u \in \mathcal{X}_\infty$ and $\lambda \in \mathbb{R} \setminus  [-C_0,C_0] + i(0,\infty)$.

In addition, since $B$ is an elliptic boundary condition in the sense of Bessel operators \cite[Section 4.4]{gannot:bessel}, one has the following:

\begin{lem}
	The operator $P_\infty(\lambda) : \mathcal{X}_\infty \rightarrow \mathcal{L}^2(X_\infty)$ is Fredholm of index zero, and is invertible outside arbitrarily small angles about the real axis for $|\lambda|$ sufficiently large.
\end{lem}
Because only ellipticity (in the semiclassical sense) is used, there is no restriction on the sign of $\Im \lambda$ \cite[Section 2.2]{gannot:bessel}. As a corollary of \eqref{eq:infinityestimate}, the operator $R_\infty(\lambda) := P_\infty(\lambda)^{-1}$ exists for $\lambda \in \mathbb{R}\setminus[-C_0,C_0] + i(0,\infty)$. In terms of the semiclassical rescaling, there exists $C>0$ such that for each $[a,b] \subseteq (0,\infty)$,
\begin{equation} \label{eq:infinityinvertible}
\| R_{h,\infty}(z) \|_{\mathcal{L}^2(X_\infty)\rightarrow \Sob^1_h(X_\infty)} \leq \frac{C}{\Im z}
\end{equation}
for $z \in [a,b] + i(0,\infty)$ and $h$ sufficiently small.

\subsection{Construction of a global approximate inverse} \label{subsect:approximateglobal}

Combining the results of Sections \ref{subsect:approximateevent}, \ref{subsect:approximateinfinity}, it is now possible to construct a global approximate inverse for $P(\lambda)$ on $X_\delta$. Choose a smooth partition of unity $\chi_1 + \chi_2 = 1$ and functions $\psi_1,\psi_2$ on $X_\delta$ such that
\begin{itemize} \itemsep6pt
	\item $\supp \chi_1\cup \supp \psi_1 \subseteq \{ r > R/2 \}$ and $\supp \chi_2\cup \supp \psi_2 \subseteq \{ r < R\}$,
	\item $\psi_1 = 1$ near $\supp \chi_1$ and $\psi_2 = 1$ near $\supp \chi_2$. 
\end{itemize}
Fix $M > 0$ such that \eqref{eq:horizoninvertible} holds. Then, choose $[C_-,C_+] \subseteq (-\varkappa/2,\infty)$ such that $M \in (C_-, C_+)$ and $|C_+ - M| \geq |C_- - M|$, increasing $C_+$ if necessary.

Given an interval $[a,b] \subseteq (0,\infty)$, the operator $\mathcal{R}_{h,+}(z)$, defined in Proposition \ref{prop:nontrapping}, exists for $z \in [a,b] + ih[C_-,C_+]$. Similarly, $R_{h,+}(z_0)$ and $R_{h,\infty}(z_0)$ exist for $z_0 \in [a,b] + ih[M,C_+]$. Define
\begin{align*}
E(z,z_0) &= \psi_1 R_{h,\infty}(z_0) \chi_1\\ &+ \psi_2 \left( \mathcal{R}_{h,+}(z) - iR_{h,+}(z_0)Q \mathcal{R}_{h,+}(z) \right)\chi_2,
\end{align*}
where $z \in [a,b] + ih[C_-,C_+]$ and $z_0 \in [a,b] + ih[M,C_+]$; here $Q$ is the absorbing operator from Section \ref{subsect:approximateevent}. This is a well defined operator $\mathcal{L}^2(X_\delta) \rightarrow \mathcal{X}(X_\delta)$ in view of the cutoffs, and it is holomorphic in $z$ for each $z_0$.

Apply $P_h(z)$ to $E(z,z_0)$ on the left: the first term yields
\begin{equation} \label{eq:K1}
\chi_1 + \psi_1 \left( P_{h,\infty}(z) - P_{h,\infty}(z_0) \right)R_{h,\infty}(z_0)\chi_1 + [P_h(z),\psi_1]R_{h,\infty}(z_0)\chi_1,
\end{equation} 
while the second term yields
\begin{multline} \label{eq:K2}
\chi_2 - i\psi_2 \left(P_{h,+}(z) - P_{h,+}(z_0) \right) R_{h,+}(z_0) Q \mathcal{R}_{h,+}(z)\chi_2  \\ + [P_h(z),\psi_2]\left( \mathcal{R}_{h,+}(z) - iR_{h,+}(z_0)Q \mathcal{R}_{h,+}(z) \right)\chi_2.
\end{multline}
Adding \eqref{eq:K1}, \eqref{eq:K2}, one has
\[
P_h(z) E(z,z_0) = I + K_1(z,z_0) + K_2(z,z_0) + K_3(z,z_0) + K_4(z,z_0),
\]
where $K_1,\,K_2$ are the second and third terms in \eqref{eq:K1}, and $K_3,\,K_4$ are the second and third terms in \eqref{eq:K2}. Also let $K = K_1 + K_2 + K_3 + K_4$.
\begin{lem} \label{lem:commutator}
	There exist compactly supported pseudodifferential operators 
	\[
	A(z) \in h\Psi^{-1}_{h}(X_+),\quad  B(z) \in h\Psi_{h}^{-\infty}(X_+)
	\]
	 depending smoothly on $z$ such that
	\[
	[P_h(z),\psi_2]\mathcal{R}_{h,+}(z) = A(z) + B(z)\mathcal{R}_{h,+}(z)
	\]	
	for $z \in [a,b] + ih[C_-,C_+]$.
\end{lem}
\begin{proof}
	This is a basic consequence of the semiclassical calculus. The commutator $[P_{h,+}(z),\psi_2]$ has coefficients supported near $\supp d\psi_2$, and if $R$ is sufficiently large, then $P_{h,+}(z)$ is elliptic near $\supp d\psi_2$ lifted to fiber infinity $\partial \OL{T^*}X_+$ by the projection $\partial \OL{T^*} X_+ \rightarrow X_+$. Choose $\varphi,\varphi' \in C_c^\infty(X_+)$ satisfying $\varphi = 1$ near $\supp d\psi_2$ and $\varphi'= 1$ near $\supp \varphi$. By choosing $\varphi'$ with sufficiently small support it may be assumed that $P_{h,+}(z)$ is elliptic near $\supp \varphi$ lifted to $\partial\OL{T^*}X_+$. By the parametrix construction \cite[Proposition E.31]{zworski:resonances}, there exist properly supported operators 
	\[
	F(z) \in \Psi^{-2}_h(X_+), \quad Y(z) \in h^\infty\Psi^{-\infty}_h(X_+)
	\]
	and a compactly supported $G \in \Psi^{\mathrm{comp}}_{h}(X_+)$ such that
	\[
	\varphi = F(z) \varphi' P_{h,+}(z) + Y(z) + G
	\]
	Here the role of $G$ is to ensure that $P_{h,+}(z)$ is semiclassically elliptic on $\mathrm{WF}_h(\varphi - G)$, while $Y(z)$ is the usual parametrix remainder. The operators $F(z), \, Y(z)$ may be chosen to depend smoothly on $z$, and $G$ can be chosen uniformly for $z$ in a compact set. Then,
	\begin{align*}
	[P_h(z),\psi_2]\mathcal{R}_{h,+}(z) & =  [P_h(z),\psi_2]\varphi\mathcal{R}_{h,+}(z) \\ &= [P_h(z),\psi_2]\left(F(z)\varphi' + (Y(z)+G)\mathcal{R}_{h,+}(z)\right).
	\end{align*}
	Since the commutator lies in $h\Psi^{1}_h(X_+)$, it suffices to define $A(z) = [P_h(z),\psi_2]F(z)\varphi'$ and $B(z) = [P_h(z),\psi_2](Y(z) + G)$.
\end{proof}

\begin{lem} \label{lem:blackbox}
	If $M > 0$ is sufficiently large, then there exists $h_0>0$ such that the following hold for $z \in [a,b] + ih[C_-,C_+]$ and $z_0 \in [a,b]+ih[M,C_+]$.
	\begin{enumerate} \itemsep6pt
		\item $K(z,z_0): \mathcal{L}^2(X_\delta) \rightarrow \mathcal{L}^2(X_\delta)$ is compact.
		\item $I + K(z_0,z_0)$ is invertible for $h \in (0,h_0)$.
	\end{enumerate}
\end{lem}
\begin{proof}
To prove compactness, consider each term in $K(z,z_0)$ separately. For $K_1, K_2$, the operator $P_{h,\infty}(z)$ is an elliptic Bessel operator (in the non-semiclassical sense) \cite[Section 4.4]{gannot:bessel}. Furthermore, if $\nu \in (0,1)$, then the boundary operator $B$ is elliptic in the sense of \cite[Section 4.4]{gannot:bessel}. First, observe that the difference $P_{h,\infty}(z) - P_{h,\infty}(z_0)$ is of first order. By elliptic regularity \cite[Theorem 1]{gannot:bessel} for Bessel operators, the space $\mathcal{X}_\infty$ has one order of regularity higher than $\Sob^1_h(X_\infty)$, so that
\[
\left( P_{h,\infty}(z) - P_{h,\infty}(z_0) \right) R_{h,\infty}(z_0): \mathcal{L}^2(X_\infty) \rightarrow \Sob^1_h(X_\infty)
\]
is bounded. The inclusion $\Sob^1_h(X_\infty) \hookrightarrow \mathcal{L}^2(X_\infty)$ is compact, which shows that 
\[
K_1(z,z_0) :\mathcal{L}^2(X_\delta) \rightarrow  \mathcal{L}^2(X_\delta) 
\]
is compact. A similar argument also shows that $K_2(z,z_0)$ is compact.

 For $K_3,K_4$, each of the terms containing $Q$ are compact since any compactly supported operator in $\Psi_{h}^{\mathrm{comp}}(X_+)$ is compact. It remains to consider the commutator term $[P_h(z),\psi_2]\mathcal{R}_{h,+}(z)$, where compactness follows from Lemma \ref{lem:commutator}.

 To prove the invertibility statement, notice that for $z = z_0$,
 \begin{equation} \label{eq:atpoint}
 K(z_0,z_0) = [P_h(z_0),\psi_1]R_{h,\infty}(z_0)\chi_1 + [P_h(z_0),\psi_2]R_{h,+}(z_0)\chi_2.
 \end{equation}
 As first order operators, the commutators are of order $\mathcal{O}(h+h \Im z_0)$. By choosing $M>0$ sufficiently large and applying \eqref{eq:horizoninvertible}, \eqref{eq:infinityinvertible}, the operator norm of $K(z_0,z_0)$ is of order $\mathcal{O}(M^{-1} + h)$, hence $I+K(z_0,z_0)$ is invertible by Neumann series for $h$ sufficiently small.
\end{proof}

From now on it will be assumed that $M > 0$ is chosen sufficiently large so that Lemma \ref{lem:blackbox} holds. This can always be achieved \emph{before} selecting $C_\pm$ since \eqref{eq:atpoint} does not involve the operator $\mathcal{R}_{h,+}(z)$.

 Since $K(z,z_0)$ is compact and $I+K(z_0,z_0)$ is invertible for an appropriate choice of $z_0$ with $h$ small, it follows that $(I+K(z,z_0))^{-1}: \mathcal{L}^2(X_\delta)\rightarrow \mathcal{L}^2(X_\delta)$ is a meromorphic family of operators. If $(I+K(z,z_0))^{-1}$ exists, then $P_h(z) : \mathcal{X}(X_\delta) \rightarrow \mathcal{L}^2(X_\delta)$ has a right inverse given by $E(z,z_0)(I+K(z,z_0))^{-1}$. In that case $P_h(z)$ is invertible, since it is of index zero by Theorem \ref{theo:QNF}. Analytic continuation then shows that
\[
R_h(z) = E(z,z_0)(I+K(z,z_0))^{-1}.
\]
Furthermore, any pole of $R_h(z)$ is also a pole of $(I+K(z,z_0))^{-1}$.

\subsection{Singular values} \label{subsect:singular}

In order to prove \eqref{prop:exponentialbound} of Proposition \ref{prop:exponentialbound}, one must bound
\[
\| R_h(z) \|_{\mathcal{L}^2(X_\delta) \rightarrow \mathcal{L}^2(X_\delta)} \\ \leq \| E(z,z_0) \|_{\mathcal{L}^2(X_\delta) \rightarrow \mathcal{L}^2(X_\delta)} \| (I+K(z,z_0))^{-1} \|_{\mathcal{L}^2(X_\delta) \rightarrow \mathcal{L}^2(X_\delta)}.
\]
Using \eqref{eq:horizoninvertible}, \eqref{eq:infinityinvertible} and Lemma \ref{lem:nontrapping}, the operator norm of $E(z,z_0)$ is of order $\mathcal{O}(h^{-2})$, which will be harmless compared to the exponentially growing bound on the norm of $(1+K(z,z_0))^{-1}$. 
\begin{lem} [{\cite[Theorem V.5.1]{gohberg1969introduction}}] \label{lem:gohberg}
	If $\mathcal{Z}$ is a Hilbert space, then 
	\[
	\|(I+A)^{-1}\|_{\mathcal{Z}\rightarrow \mathcal{Z}} \leq \frac{\det(I+|A|)}{|\det(I+A)|}
	\]
	for any operator $A:\mathcal{Z} \rightarrow\mathcal{Z}$ of trace class.
\end{lem}

Lemma \ref{lem:gohberg} cannot be applied directly to $I+K(z,z_0)$ since $K(z,z_0)$ is not of trace class. Instead, $K(z,z_0)$ lies in a Schatten $p$-class for some $p>0$. For a compact operator $A: \mathcal{Z}_1 \rightarrow \mathcal{Z}_2$ between Hilbert spaces, let $s_j(A) = s_j(A; \mathcal{Z}_1,\mathcal{Z}_2), \, j \in \mathbb{N}_{\geq 1} $ denote its singular values counting multiplicity, listed in decreasing order.

\begin{lem} \label{lem:singular values} There exists $C>0$ such that the singular values of $K(z,z_0)$ satisfy
	\[
	s_j(K(z,z_0)) \leq Ch^{-3}\,j^{\,-1/3}
	\]
	uniformly for $z \in [a,b]+ih[C_-,C_+]$ and $z_0 \in [a,b] + ih[M,C_+]$.
\end{lem}
\begin{proof}
	By Fan's inequality $s_{i+j-1}(A+B) \leq s_i(A) + s_j(B)$ applied repeatedly,
	\[
	s_j(K(z,z_0)) \leq \sum_{i=1}^4 s_{\lfloor j/4\rfloor}(K_i(z,z_0)).
	\]
	Since the operator norm of $R_{h,\infty}(z_0) : \mathcal{L}^2(X_\infty) \rightarrow \Sob^1_h(X_\infty)$ is of order $\mathcal{O}(h^{-1})$,
	\begin{align*}
\| R_{h,\infty}(z_0)u\|_{\mathcal{L}^2(X_\infty)} + \| P_{h,\infty}(z_0)R_{h,\infty}(z_0)u \|_{\Sob^1(X_\infty)} < Ch^{-1} \| u \|_{\mathcal{L}^2(X_\infty)}
	\end{align*}
	as well. Again by elliptic regularity for Bessel operators, this implies that $K_1(z,z_0) : \mathcal{L}^2(X_\delta) \rightarrow {\Sob}^1_h(X_\infty)$ is of order $\mathcal{O}(h^{-1})$. The inclusion ${\Sob}^1_h(X_\infty) \hookrightarrow \mathcal{L}^2(X_\infty)$ has singular values bounded by $Ch^{-1}j^{-1/3}$ \cite[Appendix B]{gannot:bessel}, so by standard properties of singular values, 
	\[
	s_j(K_1(z,z_0)) \leq Ch^{-2}j^{-1/3}.
	\]
	The same argument applies to $K_2(z,z_0)$, so accounting for the extra power of $h$ coming from the commutator,
	\[
	s_j(K_2(z,z_0)) \leq Ch^{-1}j^{-1/3}.
	\]
	The terms $K_3,K_4$ can be handled similarly, using that the inclusion ${H}^1_h(X_+) \hookrightarrow L^2(X_+)$ has singular values bounded by $Ch^{-1}j^{-1/3}$. The norm of $K_3(z,z_0) : \mathcal{L}^2(X_\delta)\rightarrow H^1_h(X_+)$ is of order $\mathcal{O}(h^{-2})$, so the singular values of $K_3$ give the $h^{-3}$ dependence as in the statement of the lemma. The only term that requires extra care is $[P_{h}(z),\psi_2]\mathcal{R}_{h,+}(z)$ in $K_4$, but by using Lemma \ref{lem:commutator} this operator is seen to map $L^2(X_+) \rightarrow  {H}_h^1(X_+)$ with operator norm of order $\mathcal{O}(1)$.	
\end{proof}

It follows from Lemma \ref{lem:singular values} that $K(z,z_0)^4$ is of trace class, and using Fan's inequality $s_{i+j-1}(AB) \leq s_{i}(A)s_{j}(B)$, one has the estimate
\begin{equation} \label{eq:K^4}
s_j(K(z,z_0)^4) \leq s_{\lfloor j/4\rfloor}(K(z,z_0))^4 \leq Ch^{-12}\,j^{\,-4/3}.
\end{equation}
This is uniform for $z \in [a,b]+ih[C_-,C_+]$ and $z_0 \in [a,b]+ih[M,C_+]$. To apply Lemma \ref{lem:gohberg}, write
\begin{equation} \label{eq:geometric}
(I+K(z,z_0))^{-1} = \left(\sum_{j=0}^3 (-1)^j K(z,z_0)^j \right)(I-K(z,z_0)^4)^{-1}.
\end{equation}
The norm of $K(z,z_0) : \mathcal{L}^2(X_\delta) \rightarrow \mathcal{L}^2(X_\delta)$ is of order $\mathcal{O}(h^{-2})$, so the norm of the sum on the right hand side of \eqref{eq:geometric} is polynomially bounded in $h$. 

Note from \eqref{eq:geometric} that any pole of $(I+K(z,z_0))^{-1}$ is a pole of $(I - K(z,z_0)^4)^{-1}$, hence the poles of $R_h(z)$ are among those of $(I - K(z,z_0)^4)^{-1}$. Now Lemma \ref{lem:gohberg} is applied to $K(z,z_0)^4$.

\begin{lem} \label{lem:numerator}
	 Let $F(h)$ denote the supremum of $\det(I + |K(z,z_0)^4|)$ for $z \in [a,b]+ ih[C_-,C_+]$ and $z_0 \in [a,b] + ih[M,C_+]$. Then 
	\[
	F(h) \leq e^{Ch^{-9}}
	\]
	for some $C>0$.
\end{lem}
\begin{proof}
		The logarithm of the determinant is bounded by
		\[
		\log \det(I+|K(z,z_0)^4|) = \sum_{j\geq 1}\log(1 + s_j(K(z,z_0)^4)) \leq \sum_{j\geq 1} \log(1 + Ch^{-12}j^{-4/3})
		\]
		according to \eqref{eq:K^4}. As the terms in the latter sum decrease with $j$,
		\[
		\log \det(I+|K(z,z_0)^4|) \leq \int_0^\infty \log(1+Ch^{-12} x^{-4/3})\,dx \leq Ch^{-9}
		\]
		after making the change of variables $y = h^9 x$.
\end{proof}

\noindent The next step is to bound $|f(z,z_0)|$ from below, where $ f(z,z_0) = \det(I- K(z,z_0)^4)$.

\begin{lem} \label{lem:f(z,z_0)}
	The function $f(z,z_0)$ has the following properties.
	\begin{enumerate} \itemsep6pt
		\item $|f(z,z_0)| \leq F(h)$.
		\item $f(z_0,z_0) \neq 0$, and moreover $|f(z_0,z_0)| \geq e^{-Ch^{-9}}$ for some $C>0$. 
	\end{enumerate}
\end{lem}
\begin{proof}
\begin{inparaenum}[1)] \item The estimate $|f(z,z_0)| \leq F(h)$ follows from Weyl convexity inequalities \cite[Proposition B.2.4]{zworski:resonances}.
	
	\item As in the proof of Lemma \ref{lem:blackbox}, the norm of $K(z_0,z_0)^4$ is of order $\mathcal{O}((M^{-1}+h)^4)$, so $I - K(z_0,z_0)^4$ is invertible, and
	\begin{equation} \label{eq:amongpoles}
	(I - K(z_0,z_0)^4)^{-1} = I + K(z_0,z_0)^4(I - K(z_0,z_0)^4)^{-1}.
	\end{equation}
	Arguing as in Lemma \ref{lem:numerator},
	\[
	\det(I + | K(z_0,z_0)^4(I + K(z_0,z_0)^4)^{-1}|) \leq e^{Ch^{-9}},
	\]
	which gives $|f(z_0,z_0)| \geq e^{-C h^{-9}}$.  \end{inparaenum}
\end{proof}

The proof of Proposition \ref{prop:exponentialbound} can now be finished using the following lemma of Cartan \cite[Theorem 11]{levin1964distribution}:

\begin{lem}  \label{lem:cartan}
	Suppose that $g(z)$ is holomorphic in a neighborhood of a disk $B(z_0,R)$, and $g(z_0) \neq 0$. Let $\{ w_j \}$ denote the zeros of $g(z)$ in $B(z_0,R)$ for $j = 1,\ldots,n(z_0,R)$. Given any $r\in (0,R)$ and $\rho>0$,
	\[
	\log |g(z)| - \log |g(z_0)| \geq -\frac{2r}{R-r} \log \left( \sup_{|z-z_0| < R} |g(z)| \right) - n(z_0,R) \log\left(\frac{R+r}{\rho}\right)
	\]
	for $z \in B(z_0,r) \setminus \bigcup_j B(w_j,\rho)$. 
	\end{lem}

Lemma \ref{lem:cartan} will be applied to the function $z\mapsto f(z,z_0)$ and disks of radius proportional to $h$. This requires a bound on the number of zeros of $f(z,z_0)$ in disks of the form $B(z_0,Rh)$. 

As noted at the beginning of Section \ref{subsect:approximateglobal}, it may always be assumed that $|C_+ - M | \geq |C_- - M|$. Recall the definition of $\Omega(h)$ from \eqref{eq:omega}. Then given $\varepsilon > 0$, there exists $M' \geq M$ and $R>0$ such that the union of all disks $B(w,Rh)$ with $w \in [a,b] + ih[M,M']$ covers $\Omega(h)$ and is contained in $\Omega_\varepsilon(h) = [a-\varepsilon,b+\varepsilon] + ih[C_- - \varepsilon, C_+ + \varepsilon]$. If both $\varepsilon > 0$ and $h> 0$ are sufficiently small, then 
\[
\Omega_{3\varepsilon}(h) \subseteq [a', b'] + i h [C_-',C_+'],
\]
where $[a',b'] \subseteq (0,\infty)$ and $[C_-',C_+'] \subseteq (-\varkappa/2,\infty)$. Applying Lemmas \ref{lem:numerator}, \ref{lem:f(z,z_0)} to this larger rectangle shows that certainly $|f(z,z_0)| \leq e^{Ch^{-9}}$ for $z_0 \in [a,b] + ih [M,M']$ and $z \in B(z_0, (R+2\varepsilon)h)$.

\begin{lem} \label{lem:zeros}
	Let $z_0 \in [a,b] + ih[M,M']$. Then there exists $C>0$ such
	\[
	n(z_0,(R+\varepsilon)h) \leq Ch^{-9},
	\]
	uniformly in $z_0$, where $n(z_0,(R+\varepsilon)h)$ is the number of zeros of $f(z,z_0)$ in $B(z_0,(R+\varepsilon)h)$.
\end{lem}
\begin{proof}
	 By Jensen's formula, the number of zeros $n(z_0,\rho)$ of $f(z,z_0)$ within $B(z_0,\rho)$ satisfies
	 \begin{align*} \label{eq:jensen}
	 \int_0^{h(R+2\varepsilon)} \frac{n(z_0,\rho)}{\rho} d\rho &= \frac{1}{2\pi} \int_0^{2\pi} \log |f(z_0 + h(R+2\varepsilon)e^{i\theta},z_0)| \,d\theta - \log|f(z_0,z_0)| \\&\leq Ch^{-9}.
	 \end{align*}
	 Therefore the number of zeros in a disk $n(z_0,(R+\varepsilon)h)$ is estimated by
	 \[
	 \frac{\varepsilon}{R+2\varepsilon} n(z_0, (R+\varepsilon)h) \leq \int_{h(R+\varepsilon)}^{h(R+2\varepsilon)} \frac{n(z_0,\rho)}{\rho} \, d\rho \leq Ch^{-9}. \qedhere
	 \] \end{proof}

Combining Lemma \ref{lem:cartan} with Lemma \ref{lem:zeros} shows that for any function $0 < S(h) =o(h)$ and $z_0 \in [a,b] + ih[M,M']$,
\begin{equation} \label{eq:flowerbound}
|f(z,z_0)| \geq \exp(-Ch^{-9}\log(1/S(h))), \quad z \in B(z_0, Rh) \setminus \bigcup_j B(w_j,S(h)),
\end{equation}
where $\{ w_j \}$ are the zeros of $f(z,z_0)$ in $B(z_0, (R+\varepsilon)h)$.

\begin{proof} [Proof of Proposition \ref{prop:exponentialbound}]
	Combine Lemma \ref{lem:numerator} with \eqref{eq:flowerbound}. This shows that for any $z_0 \in [a,b] + ih[M,M']$, 
	\begin{equation} \label{eq:expbound1}
	\| R_h(z) \|_{\mathcal{L}^2(X_\delta) \rightarrow \mathcal{L}^2(X_\delta)} \leq \exp(Ch^{-9}\log(1/S(h)))
	\end{equation}
	for 
	\[ z \in B(z_0,Rh) \setminus \bigcup_j B(w_j,S(h)),
	\] where $\{ w_j \}$ are the zeros of $f(z,z_0)$ in $B(z_0, (R+\varepsilon)h)$. If $w_j$ is not a pole of $R_h(z)$, then apply the maximum principle to the holomorphic operator-valued function $R_h(z)$ on $B(w_j,S(h))$ to see that \eqref{eq:expbound1} is valid on $B(w_j,S(h))$ as well. Thus \eqref{eq:expbound1} holds for $z \in B(z_0,Rh) \setminus \bigcup_j B(z_j,S(h))$, where now $\{ z_j \}$ denote the poles of $R_h(z)$ in  $B(z_0, (R+\varepsilon)h)$. The disks $B(z_0,Rh)$ for $z_0 \in [a,b] + ih[M,M']$ cover $\Omega(h)$, whence the result follows.\end{proof}

\begin{rem} Lemma \ref{lem:zeros} also gives a polynomial bound on the number of poles of $R_h(z)$ in $\Omega(h)$, albeit not likely an optimal one. For each fixed $z_0$, the poles of $R_h(z)$ are among the zeros of $f(z,z_0)$. The region $\Omega(h)$ can be covered by at most $\mathcal{O}(h^{-1})$ disks of radius $Rh$ with centers in $[a,b] + ih[M,M']$. According to Lemma \ref{lem:zeros}, $R_h(z)$ has at most $\mathcal{O}(h^{-9})$ poles in each of these disks, so altogether $R_h(z)$ has at most $\mathcal{O}(h^{-10})$ poles in $\Omega(h)$.
\end{rem}

\section{QNFs in the upper half-plane} \label{sect:uppermodes}

To complete the proof of Theorem \ref{theo:maintheorem} it remains to prove Propositions \ref{prop:uppermodes}, \ref{prop:realQNF}. In the former case the proof is very similar to that of Lemmas \ref{lem:invertible1} and \ref{lem:twistedinvertible}. The twisted stress-energy tensor $\widetilde{\mathbb{T}} =\widetilde{\mathbb{T}}[v]$ is defined as in \eqref{eq:twisteddivergence} using the metric $g$; the twisting function is again $q = r^{\nu -3/2}$, and $Q = q^{-1}\Box_g(q)$.

The future pointing Killing generator $K$ of the null surface $\{ r= r_+\}$ can be normalized by requiring $Kt^\star = 1$, in which case
\[
K = T + \frac{a}{r_+^2+a^2} \Phi.
\]
Let $d\sigma$ denote the measure induced on $H_0 = \OL{X}_0 \cap \{ r= r_+\}$. With the above normalization, the analogue of \eqref{eq:twisteddivergence} on $X_0$ has the form
 \begin{multline*} 
 \partial_{t^\star}\int_{X_0} \widetilde{\mathbb{T}}(W,\bar{N}_t) \, r^{-1}dS_t - \int_{Y} \Re \left(\gamma_- W v \cdot \gamma_+\OL{v} \right) \bar{A}\, d\bar{\mathcal{K}}_0 \\ = - \int_{H_0} \widetilde{\mathbb{T}}(W,K) \, d\sigma +
 \int_{X_0} \left( \Re F \cdot W\OL{v} \right) A \,dS_t,
 \end{multline*}
 where $W$ is a Killing field such that $Wr= 0$ and $A = g^{-1}(dt^\star,dt^\star)^{-1/2}$ is the lapse function. 
 
As before, $F = (\Box_g + \nu^2 - 9/4)v$. This identity is applied with the vector field $W = K$. The contribution from the horizon is the integral of
 \[
 \widetilde{\mathbb{T}}(K,K) = |Kv|^2 \geq 0,
 \]
 which may be dropped in view of its nonnegativity. 
 
 If $|a| < r_+^2$, then $K$ is everywhere timelike on $\mathcal{M}_0$. In that case Proposition \ref{prop:uppermodes} would follow from the same proof as in Lemma \ref{lem:twistedinvertible}; the only difference is that coercivity of the derivative transverse to the horizon degenerates at the horizon. This does not affect the final result since Proposition \ref{prop:uppermodes} only involves an $L^2$ bound. Without the timelike assumption, a direct calculation in terms of the metric coefficients gives
 \begin{align*}
 \frac{2}{A} \, \widetilde{\mathbb{T}}(\partial_{t^\star},N_t) &= g^{t^\star t^\star} |T v |^2 - g^{rr}|\tilde\partial_r v|^2  - 2g^{r\phi^\star} \Re ( \Phi v \cdot \tilde \partial_{r}\OL v ) \\ &- g^{\phi^\star \phi^\star } |\Phi v |^2 - g^{\theta \theta  } |\partial_{\theta} v |^2 + (\nu^2 - 9/4 + Q)|v|^2,
 \end{align*}
 as well as
 \[
 \frac{1}{{A}} \, \widetilde{\mathbb{T}}(\partial_{\phi^\star},N_t) = g^{t^\star t^\star}\Re ( \Phi v \cdot T \OL v) + g^{t^\star r} \Re ( \Phi v \cdot \tilde \partial_{r}\OL v ) + g^{t^\star \phi^\star} |\Phi v|^2.
 \]
 An important preliminary observation is that the coefficient of $\Re(\Phi v \cdot \tilde \partial_r \bar v)$ in $\widetilde{\mathbb{T}}(K,N_t)$ is proportional to 
 \[
 a g^{t^\star r} - (r_+^2+a^2)g^{\phi^\star r},
 \] 
 hence vanishes at $\mathcal{H}_0$. If $v = e^{-i\lambda t^\star}u$ with $u \in \mathcal{D}_k(X_0)$, then the stress-energy tensor associated to $v$ can be written as
 \begin{align*}
 e^{-2(\Im \lambda) t^\star}\,\widetilde{\mathbb{T}}(K,\bar{N}_t) &= F_{1} \,|\lambda|^2|u|^2 + F_{2} \,|\tilde{\partial}_r u|^2 + F_3 \,|\partial_{\theta}u|^2  \\ &+ E_1\cdot k \Im(u \cdot \tilde{\partial}_r u) + E_2(\lambda,k)\,|u|^2.
 \end{align*}
 Here the coefficient functions $F_1,F_2,F_3,E_1$ are independent of $k$ and $\lambda$, while $E_2(\lambda,k)$ depends on both $k$ and $\lambda$. In terms of their behavior near $H_0$ and $Y$, the following properties are satisfied:
 \begin{itemize} \itemsep6pt
 	\item $F_1 > 0$ and $F_3 > 0$ on $X_0 \cup H_0 \cup Y$, while $F_2\geq 0$ vanishes simply at $H_0$ but is otherwise positive. 
 	\item  There is $C>0$ such that $1/C \leq F_1 \leq C$ on $X_0 \cup H_0 \cup Y$ and $r^4/C \leq  F_2 \leq Cr^4$ for large $r$.
  \item $E_1$ vanishes simply at $H_0$ and $E_1 = \mathcal{O}(r)$ on $X_0 \cup H_0 \cup Y$.
 
 \item $E_{2}(\lambda,k) = \mathcal{O}(1+|\lambda|)$ on $X_0 \cup H_0 \cup Y$ for each fixed $k$, uniformly in $|\lambda|$.
 \end{itemize}
 From these properties it is clear that $F_1 |\lambda|^2|u|^2$ can be used to absorb $E_2(\lambda,k) |u|^2$ for large values of $|\lambda|$. In addition, 
 \[
 -2E_1\cdot  k \Im (u\cdot \tilde \partial_r u) \leq  \delta|E_1|^2 |\,\tilde\partial_r u|^2 + \delta^{-1}|k|^2|u|^2, 
 \]
 which can be absorbed by a combination of $F_1 |\lambda|^2|u|^2$ and $F_2|\tilde\partial_r u|^2$ for sufficiently small $\delta >0$ and large $|\lambda|$.

\begin{proof} [Proof of Proposition \ref{prop:uppermodes}]
Let  $v= e^{-i\lambda t^\star}u$ for $u \in \mathcal{D}'_k(X_0)$. As pointed out in the preceding paragraph, the term
\[
\partial_{t^\star} \int_{X_0} \widetilde{T}(W,\bar N_t) \, r^{-1} dS_t
\]
controls $\Im \lambda |\lambda|^2 \| u \|^2_{\mathcal{L}^2(X_0)}$ for $\Im \lambda > 0$ and $|\lambda|$ sufficiently large. The proof is now finished as in Lemma \ref{lem:twistedinvertible}; the only additional observation is that since $k \in \mathbb{Z}$ is real, it does not contribute to the boundary integral (compare this with the expression \ref{eq:boundarycontribution}).
\end{proof}

Finally, consider Proposition \ref{prop:realQNF}, which asserts that apart from some exceptional values, there can be no QNFs on the real axis at a fixed axial mode. The proof is a boundary-pairing argument, well known in scattering theory \cite[Section 2.3]{melrose1995geometric}. The first part of the proof is adapted from \cite[Lemma A.1]{warnick:2015:cmp}, which is particularly useful in light of the geometry at both $H_0$ and $Y$. For another approach in the relativistic setting, see \cite[Section 3.2]{hintz2015asymptotics}.

\begin{proof}[Proof of Proposition \ref{prop:realQNF}]
	First, recall that elements $u \in \mathcal{X}(X_\delta)$ in the kernel of $P(\lambda)$ lie in $C^\infty(X_\delta \cup H_\delta)$, provided $\Im \lambda > -\varkappa/2$ \cite[Proposition 6.2]{gannot:2014:kerr}. Now suppose that $u \in \mathcal{X}(X_\delta) \cap \mathcal{D}'_k(X_\delta)$ satisfies $P(\lambda)u = 0$. If $\lambda \in \mathbb{R}$, then \cite[Corollary 3.2]{warnick:2015:cmp} (which is just an application of the divergence theorem) shows that
	\[
	\left((r_+^2 + a^2)\lambda - ak\right)\int_{H_0} |u|^2\, d\sigma = 0.
	\]
	This holds true for either Dirichlet or Robin boundary conditions with $\beta \in C^\infty(Y;\mathbb{R})$.  Therefore $u$ vanishes on $H_0$ if $(r_+^2 + a^2)\lambda \neq ak$. As in \cite[Lemma A.1]{warnick:2015:cmp} or \cite[Proposition 3.6]{hintz2015asymptotics}, one would like to apply some type of unique continuation result to conclude that $u$ must in fact vanish everywhere. This is known to be a difficult problem in view of possible trapping within the ergoregion where $P(\lambda)$ fails to be elliptic \cite{ionescu2009uniqueness}. To work around this, define the Riemannian metric
	\begin{equation} \label{eq:hmetric}
	\tilde{g} = \frac{1}{\Delta_\theta} \,d\theta^2 + \frac{\Delta_\theta\sin^2\theta}{(1-a^2)^2} \,(d\phi^\star)^2,
	\end{equation}
	and let $\Delta_{\tilde{g}}$ denote its nonnegative Laplacian. Observe that the difference between $P(\lambda)$ and the operator $\Delta_rD_r^2 + \Delta_{\tilde{g}}$ is of first order modulo the second order term $2a(1-a^2)D_rD_{\phi^\star}$. Set
	\[
	\widetilde{P}(\lambda,k) = \Delta_r D_r^2 + \Delta_{\tilde{g}} + e^{-ik\phi^\star}(P(\lambda)-\Delta_r D_r^2 - \Delta_{\tilde g})e^{ik\phi^\star}.
	\]
	Thus $\widetilde{P}(\lambda,k)$ is elliptic on $\{r > r_+\}$, and furthermore $\widetilde{P}(\lambda,k)u = P(\lambda)u$ for each $u \in \mathcal{D}'_k(X_0)$. Next, define the quantity 
	\[
	s(\lambda,k) = 2(1-a^2)(ak - (r_+^2+a^2)\lambda),
	\]
	which by assumption is real valued. In terms of the new radial coordinate $\rho = r-r_+$,
	\[
	\rho\tilde{P}(\lambda,k) = -\Delta_r'(r_+)(\rho \partial_\rho)^2 - is(\lambda,k)\rho \partial_\rho
	\]
	modulo a differential operator which maps $\rho^m C^\infty(X_0 \cup H_0) \rightarrow \rho^{m+1} C^\infty(X_0 \cup H_0)$ for each $m$. In the inductive step, assume $u = \rho^m C^\infty(X_0\cup H_0)$, where $m \geq 1$. Then $\rho\tilde{P}(\lambda,k)u=0$ implies 
	\[
	(\Delta_r'(r_+)m^2 +is(\lambda,k)m)u \in \rho^{m+1}C^\infty(X_0 \cup H_0).
	\]
	Since the coefficient of $u$ is never zero for $m \geq 1$, it follows by induction that $u$ in fact vanishes to infinite order at $H_0$. This is just an argument about the indicial roots of $\rho \tilde{P}(\lambda,k)$, where the latter can be replaced more generally by a $0$-differential operator, see the discussion in \cite[Proposition 3.6]{hintz2015asymptotics}. As in \cite{hintz2015asymptotics,warnick:2015:cmp}, it now follows by a unique continuation argument that $u$ must vanish near $H_0$ (see  \cite[Theorem 2]{roberts1980uniqueness}, and in particular, \cite[Example 1]{roberts1980uniqueness}); since $\widetilde{P}(\lambda,k)$ is elliptic in the usual sense away from the boundary, $u=0$ throughout $X_0 \cup H_0$. Referring to \cite[Proposition 7.1]{gannot:2014:kerr} for a unique continuation argument across $H_0$ to $X_\delta \setminus X_0$, it follows that $u$ vanishes identically on $X_\delta$ as well.
\end{proof}

\section*{Acknowledgements}
I would like to thank Semyon Dyatlov, Peter Hintz, Andr\'as Vasy, and Maciej Zworski for their interest in the problem and many useful discussions. I am especially grateful to the anonymous referee for carefully reviewing the paper, and for suggesting several improvements to both the content and exposition.

\bibliographystyle{plain}

\bibliography{biblio}

\end{document}